\documentclass[a4paper,11pt]{article}
\usepackage[latin1]{inputenc}
\usepackage{amsfonts}
\usepackage{amsthm}
\usepackage{amsmath}
\usepackage{amsfonts}
\usepackage{latexsym}
\usepackage{amssymb}
\usepackage{dsfont}

\newtheorem{fed}{Definition}[section]
\newtheorem*{fed*}{Definition}
\newtheorem*{feds*}{Definitions}
\newtheorem{teo}[fed]{Theorem}
\newtheorem*{teo*}{Theorem}
\newtheorem{lem}[fed]{Lemma}
\newtheorem{cor}[fed]{Corollary}
\newtheorem{pro}[fed]{Proposition}
\theoremstyle{definition}
\newtheorem{rem}[fed]{Remark}
\newtheorem{conj}[fed]{Conjecture}
\newtheorem*{rems*}{Remarks}

\newtheorem{exa}[fed]{Example}

\newtheorem{nota}[fed]{Notations}

\def\ID{Let $(\cF_0 \coma \ca)$ be initial data
for the CP} 

\def\coma{\, , \, }

\def\py{\peso{and}}
\newcommand{\peso}[1]{ \quad \text{ #1 } \quad }

\def\n0{n_{ \text{\rm \tiny o}}}

\newcommand{\IN}[1]{\mathbb {I} _{#1}}
\def\In{\mathbb {I} _n}
\def\IM{\mathbb {I} _m}

\def\QEDP{\tag*{\QED}}

\def\bce{\begin{center}}
\def\ece{\end{center}}

\DeclareMathOperator{\FP}{FP\,}

\def\cD{\mathcal D}

\def\py{\peso{and}}
\def\rk{\text{\rm rk}}
\def\noi{\noindent}
\def\cF{\mathcal F}
\def\cG{\mathcal G}
\def\QED{\hfill $\square$}
\def\EOE{\hfill $\triangle$}
\def\EOEP{\tag*{\EOE}}
\def\uno{\mathbb{1}}
\def\bm{\left[\begin{array}}
\def\em{\end{array}\right]}
\def\ben{\begin{enumerate}}
\def\een{\end{enumerate}}
\def\bit{\begin{itemize}}
\def\eit{\end{itemize}}
\def\barr{\begin{array}}
\def\earr{\end{array}}
\def\igdef{\ \stackrel{\mbox{\tiny{def}}}{=}\ }

\def\la{\lambda}
\def\al{\alpha}
\def\N{\mathbb{N}}
\def\R{\mathbb{R}}
\def\C{\mathbb{C}}
\def\I{\mathbb{I}}
\def\Z{\mathbb{Z}}
\def\cA{\mathcal{A}}

\def\T{\mathbb{T}}
\def\cR{{\cal R}}
\def\cS{{\cal S}}
\def\cT{{\cal T}}
\def\cM{{\cal M}}
\def\cB{{\cal B}}

\def\cV{{\cal V}}

\def\cW{{\cal W}}

\def\cX{\mathcal{X}}
\def\cY{\mathcal{Y}}
\def\cZ{\mathcal{Z}}

\def\orto{^\perp}
\def\inc{\subseteq}

\def\inv{^{-1}}

\def\api{\langle}
\def\cpi{\rangle}

\def\ua{^\uparrow}
\def\da{^\downarrow}

 \DeclareMathOperator{\tr}{tr}
\DeclareMathOperator{\gen}{span}

\DeclareMathOperator{\geqp}{\geqslant}
\DeclareMathOperator{\convf}{Conv (\R^+)}
\DeclareMathOperator{\convfs}{Conv_s (\R^+)}

\newcommand{\hil}{\mathcal{H}}

\newcommand{\esssup}{{\mathrm{ess}\sup}}
\newcommand{\essinf}{{\mathrm{ess}\inf}}

\newcommand{\mat}{\mathcal{M}_d(\mathbb{C})}
\newcommand{\matn}{\mathcal{M}_n(\mathbb{C})}

\newcommand{\matpos}{\mat^+}
\newcommand{\matposn}{\matn^+}

\newcommand{\matinvd}{\mathcal{G}\textit{l}\,(d)}

\def\beq{\begin{equation}}
\def\eeq{\end{equation}}

\def\pausa{\medskip\noi}

\def\const{C_\cW}

\def\duv{\cD_\cV^{SG}(\cF)}
\def\duvw{\cD_{\cV\coma w}^{SG}(\cF)}
\def\Sast{A}
\def\Bx{[\Sast]_x}
\def\Ax2{\,[ S_{E(\cF)^\#_\cV}]_x }
\def\fe{\peso{for every}}

\begin{document}
\title{Convex potentials and optimal shift generated \\ oblique duals  in shift invariant spaces}
\author{Mar\'\i a J. Benac $^{*}$, Pedro G. Massey $^{*}$, and Demetrio Stojanoff 
\footnote{Partially supported by CONICET 
(PIP 0435/10) and  Universidad Nacional de La PLata (UNLP 11X681) } \ %\hspace{19cm}
 \footnote{ e-mail addresses: mjbenac@gmail.com , massey@mate.unlp.edu.ar , demetrio@mate.unlp.edu.ar}
\\ 
{\small Depto. de Matem\'atica, FCE-UNLP
and IAM-CONICET, Argentina  }}

\date{}
\maketitle
\begin{abstract}
We introduce extensions of the convex potentials for finite frames (e.g. the frame potential defined by Benedetto and Fickus) in the framework of Bessel sequences of integer translates of finite sequences in $L^2(\R^k)$. We show that under a natural normalization hypothesis, these convex potentials detect tight frames as their minimizers. We obtain a detailed spectral analysis of the frame operators of shift generated oblique duals of a fixed frame of translates. We use this result to obtain the spectral and geometrical structure of optimal shift generated  oblique duals with norm restrictions, that simultaneously minimize every convex potential; we approach this problem by showing that the water-filling construction in probability spaces is optimal with respect to submajorization (within an appropriate set of functions) and by considering a non-commutative version of this construction for measurable fields of positive operators. 
\end{abstract}



\newpage
\tableofcontents

\section{Introduction}\label{intro}

Let $\cW$ be a closed subspace of a separable complex Hilbert space $\hil$ and let $\mathbb I$ be a finite or countable infinite set. 
%%A (possibly finite) sequence $\cF=\{f_i\}_{i\in \mathbb I}$ in $\cW$ is a frame for $\cW$ if there exist positive constants $0<a\leq b$ such that 
%%$$ 
%%a \, \|f\|^2\leq \sum_{i\in \mathbb I}|\langle f,f_i\rangle |^2\leq b\,\|f\|^2 
%%\peso{for every} f\in \cW\, . 
%%$$ 
%If $a=b$ then we say that $\cF$ is a tight frame for $\cW$.
%
%\pausa
A frame $\cF=\{f_i\}_{i\in \mathbb I}$ for $\cW$ allows for linear (typically redundant) and stable encoding-decoding schemes of vectors (signals) in $\cW$ (see Section \ref{sec defi oblique duals} for definitions and technical results). Indeed, if 
$\cV$ is a closed subspace of $\hil$ such that $\cV\oplus\cW^\perp=\hil$ (e.g. $\cV=\cW$) then  it is possible to find frames $\cG=\{g_i\}_{i\in \mathbb I}$ for $\cV$ such that 
$$f=\sum_{i\in \mathbb I}\langle f,g_i\rangle \ f_i \ , \quad \text{ for } f\in\cW\,.  $$
The representation above lies within the theory of oblique duality (see \cite{YEldar3,CE06,YEldar1,YEldar2}). 
In applied situations, it is usually desired to develop encoding-decoding schemes as above, with some additional features. In some cases, we search for schemes with prescribed properties (e.g., for which the sequence of norms $\{\|f_i\|^2\}_{i\in \mathbb I}$ as well as the spectral properties of the family $\cF$ are given in advance) leading to what is known in the literature as frame design problem (see \cite{AMRS,BF,CFMPS,CasLeo,MR08,Pot}). 
In other cases, we search for numerically robust oblique dual pairs $(\cF,\cG)$ as above, leading to what is known as optimal frame designs (\cite{BMS14,CE06,FMP,MR10,MRS13b,Telat}).

\pausa
In their seminal work \cite{BF}, Benedetto and Fickus introduced a functional defined on finite sequences of (unit norm) vectors $\cF=\{f_i\}_{i\in\In}$ (where $\In=\{1,\ldots,n\}$), the so-called {\it frame potential}, given by 
\beq\label{rep BFpot}
\FP(\cF)
=\sum_{i,\,j\,\in \In}|\api f_i\coma f_j \cpi |\,^2\ .
\eeq 
In case $\dim \hil=d\in\N$ then one of their major results shows that tight unit norm frames can be characterized as (local) minimizers of this functional, among unit norm frames. Since then, there has been interest in (local) minimizers of the frame potential within certain classes of frames, since such minimizers can be considered as natural substitutes of tight frames (see for example \cite{CKFT,MR10}).
Recently, there has been interest in the structure of minimizers of other potentials such as the so-called mean squared error (MSE) 
(see \cite{FMP,MRS13,Pot} and the references therein).  
Both the frame potential and the MSE are examples of the so-called convex potentials introduced in \cite{MR10}. It turns out that minimizers of these convex potentials share the spectral and geometrical structure of minimizers of the frame potential.  Now, it is a well known fact that in case $\cV=\cW$ then tight frames $\cF$ for $\cW$ - i.e. minimizers of convex potentials - give rise to optimal (numerically robust) dual pairs $(\cF,\cG)$. Therefore, it seems apparent that 
in the general case $\cV\oplus \cW^\perp$ the construction of robust oblique dual pairs $(\cF,\cG)$ is related with the construction of frames $\cF$ which are minimizers of convex potentials (e.g. the frame potential).

\pausa
It turns out that there is a single notion that ties all the previous problems together namely, the majorization preorder. Indeed, majorization is the key notion behind the frame design problems (see \cite{AMRS,CFMPS,CasLeo}) through natural extensions of the Schur-Horn theorem from matrix analysis (see \cite{AMRS,BoJa11,BoJa,KaWe}). Moreover, the relation between majorization and tracial inequalities with respect to convex functions allows to apply this notion in the study of convex potentials (\cite{BMS14,MR10,MRS13,MRS14,MRS13b}).
Unfortunately, the convex potentials considered in \cite{MR10} (in particular, the frame potential) can only be defined for finite frames. Hence, in the infinite dimensional context we loose a tool which have proved useful as a measure of stability for frames in finite dimensional Hilbert spaces. 

\pausa
In this paper we show that there are natural analogues of the convex potentials (and in particular, of the frame potential) in the context of Bessel sequences of integer translates $E(\cF)$ of a finite family of vectors $\cF=\{f_i\}_{i\in\In}$ that lie in a finitely generated shift invariant (FSI) subspace $\cW$ of $L^2(\R^k)$. We take advantage of the detailed structure of $E(\cF)$ as a fibered family over the torus 
$\T^k$ (see \cite{BDR,Bo,RS95}) and the theory of range functions for shift invariant (SI) subspaces and introduce the potential $P_{\varphi}^\cW(E(\cF))$ associated to a convex function $\varphi$ and $\cW$ as an integral over $\T^k$ of the corresponding potentials on the fibers (for related approaches to different problems in SI subspaces see \cite{ACHM07,HG07,HG09}). In order to verify that our definition is a natural extension of the convex potentials for finite frames we show that under natural normalization conditions, a family of integer translates $E(\cF)$ that is a tight frame for a FSI subspace $\cW$ is a minimizer of the convex potential associated to $\varphi$ and $\cW$.

\pausa
The convex potentials in FSI subspaces raise several questions related with optimal design problems. In particular, given FSI subspaces $\cW,\,\cV$ such that $\cV\oplus \cW^\perp=L^2(\R^k)$ and a finite family $\cF=\{f_i\}_{i\in\In}$ such that $E(\cF)$ is a frame for $\cW$, we consider the 
problem of designing optimal oblique duals $E(\cG)$ which are translates of a family $\cG=\{g_i\}_{i\in\In}$ in $\cV$ and such that $\cG$ satisfies the norm restrictions $\sum_{i\in\In}\|g_i\|^2\geq w$, for $w>0$.
 In order to deal with this problem we develop two new tools in the context of
frames of translates. On the one hand, we obtain what we call the fine spectral structure of 
shift generated (SG) oblique dual frames of the fixed frame $E(\cF)$, which is a detailed description of the eigenvalues of the measurable field of positive operators defined on $\T^k$ corresponding to the frame operators of SG oblique duals of $E(\cF)$. As a consequence, we derive necessary and sufficient conditions for the existence of a tight SG oblique dual $E(\cG)$ of $E(\cF)$. On the other hand, we consider the water-filling construction (both for functions in probability spaces as well as for measurable field of positive finite-rank operators) and show that this construction leads to optimal solutions of the oblique dual design problem; this is achieved by showing that the water-filling constructions are optimal with respect to majorization (considered in the general context of probability spaces) which is a result of independent interest. With these tools we completely solve the problem of designing optimal oblique dual frames with norm restriction mentioned before; it turns out that these optimal SG oblique duals are more stable than the so-called canonical oblique dual. We point out that the structure of the optimal solution is obtained in terms of a global analysis.
As a byproduct we extend the so-called Fan-Pall interlacing theorem from matrix analysis to the context of 
measurable fields of positive matrices.

\pausa
The paper is organized as follows. In section \ref{sec2} we describe several preliminary notions and facts from frame theory, SI subspaces, frames of translates and majorization theory in probability spaces. In section \ref{sec3} we introduce the convex potentials for frames of translates and show that are natural extensions of the convex potentials for finite frames. In section \ref{sec4} we recall several facts from the theory of oblique duality in FSI subspaces and obtain 
the precise value of the aliasing norm corresponding to the consistent sampling in this setting. Then, we describe the fine spectral structure of oblique duals of a fixed SG frame. Since this result depends on an extension of the Fan-Pall
interlacing theory, its proof is presented in an appendix (see section \ref{Apendixity}). In section \ref{sec applic} we study the problem of optimal design of oblique dual frames $E(\cG)$ - of a fixed finitely SG 
frame $E(\cF)$ - which satisfy certain norm restrictions. We first show that the water-filling construction for positive functions in probability spaces is optimal with respect to sub-majorization within a natural set of functions. We then construct optimal 
SG oblique duals with norm restrictions and explain the relation of our construction with a natural (non-commutative) water-filling construction for measurable field of positive finite-rank operators. The paper ends with an appendix section in which we develop the Fan-Pall interlacing theorem for measurable fields of positive matrices as well as some consequences of this result.

\section{Preliminaries}\label{sec2}

In this section we recall some basic facts related with frame theory, 
oblique duality and shift invariant (SI) subspaces of $L^2(\R^k)$. 
At the end of this section we describe majorization between functions in arbitrary probability spaces.

\def\IM{\mathbb {I} _m}
\def\ID{\mathbb {I} _d}
\def\uno{\mathds{1}}

\pausa
{\bf General Notation} 

\pausa
Throughout this work we shall use the following notation:
the space of complex $d\times d$ matrices is denoted by $\mat$ and $\matpos$ denotes the set of positive semidefinite
matrices.
$\matinvd$ is the group of invertible elements of $\mat$ and $\matinvd^+ = \matpos \cap \matinvd$. 
If $T\in \mat$, we  denote by   
$\|T\|$ its spectral norm,
by $\rk\, T= \dim R(T) $  the rank of $T$,
and by $\tr T$ the trace of $T$. 

\pausa
Given $d \in \N$ we denote by $\ID = \{1, \dots , d\} \inc \N$.
For a vector $x \in \R^m$ we denote  
by $x^\downarrow$ (resp. $x^\uparrow$) the rearrangement
of $x$ in  decreasing (resp. increasing) order. We denote by
$(\R^d)^\downarrow = \{ x\in \R^d : x = x^\downarrow\}$ the set of downwards ordered vectors, and 
$(\R^d)\ua 
=\{x\in \R^d:\ x=x\ua\}$.

\pausa
Given  $S\in \matpos$, we write $\la(S) = \la\da(S)= (\la_1(S) \coma \dots \coma \la_d(S)\,) \in 
(\R^d)^\downarrow$ the 
vector of eigenvalues of $S$ - counting multiplicities - arranged in decreasing order. 
Similarly we denote by $\la\ua(S) \in (\R^d)\ua$ the reverse ordered 
vector of eigenvalues of $S$.

\pausa
If $W\inc \C^d$ is a subspace we denote by $P_W \in \matpos$ the orthogonal 
projection onto $W$. 
Given $x\coma y \in \C^d$ we denote by $x\otimes y \in \mat$ the rank one 
matrix given by 
\beq \label{tensores}
x\otimes y \, (z) = \api z\coma y\cpi \, x \peso{for every} z\in \C^d \ .
\eeq
Note that, if $x\neq 0$, then
the projection $P_x \igdef P_{\gen\{x\}}= \|x\|^{-2} \, x\otimes x \,$.

\subsection{Frames for subspaces and oblique duality}\label{sec defi oblique duals}

In what follows $\hil$ denotes a separable complex Hilbert space and $\mathbb I$ denotes a finite or countable infinite set. 
Let $\cW$ be a closed subspace of $\hil$: recall that a sequence $\cF=\{f_i\}_{i\in \mathbb I}$ in $\cW$ is a {\it frame} for $\cW$ 
if there exist positive constants $0<a\leq b$ such that 
\beq\label{defi frame} 
a \, \|f\|^2\leq \sum_{i\in \mathbb I}|\langle f,f_i\rangle |^2\leq b\,\|f\|^2 
\peso{for every} f\in \cW\, . 
\eeq
In general, if $\cF$ satisfies the inequality to the right in Eq. \eqref{defi frame}
we say that $\cF$ is a $b$-Bessel sequence.

\pausa 
Given a Bessel sequence $\cF=\{f_i\}_{i\in \mathbb I}$ we consider its {\it synthesis operator} $T_\cF\in L(\ell^2(\mathbb I),\hil)$ given by $T_\cF((a_i)_{i\in \mathbb I})=\sum_{i\in \mathbb I} a_i\ f_i$ which, by hypothesis on $\cF$, is a bounded linear transformation. We also consider $T_\cF^*\in L(\hil,\ell^2(\mathbb I))$ called the {\it analysis operator} of $\cF$, given by $T_\cF^*(f)=(\langle f,f_i\rangle )_{i\in \mathbb I}$ and the {\it frame operator} of $\cF$ defined by $S_\cF=T_\cF\,T_\cF^*$. It is straightforward to check that
$$ 
\langle S_\cF f,f \rangle =\sum_{i\in \mathbb I}|\langle f,f_i\rangle |^2 \fe
 f\in \hil\ . 
 $$
Hence, $S_\cF$ is a positive semidefinite bounded operator; moreover, a Bessel sequence $\cF$ in $\cW$ 
is a frame for $\cW$ if and only if $S_\cF$ is an invertible operator when restricted to $\cW$ or equivalently, if the range of $T_\cF$ coincides with $\cW$.

\pausa
In order to describe oblique duality, we fix two closed subspaces $\cV,\,\cW\inc\hil$ such that $\cW^\perp \oplus \cV=\hil$, that is such that $\cW^\perp + \cV=\hil$ and $\cW^\perp \cap \cV=\{0\}$. Hence, $\cW^\perp$ is a common (algebraic) complement of $\cW$ and $\cV$.
It is well known that in this case $P_\cW|_\cV:\cV\rightarrow \cW$ is a linear bounded isomorphism so, in particular,
we see that $\dim \cV=\dim \cW$ as Hilbert spaces. Moreover, the conditions $\cW^\perp \oplus \cV=\hil$ and $\cW \oplus \cV^\perp=\hil$ are actually equivalent.

\pausa
Fix a frame $\cF=\{f_i\}_{i\in \mathbb I}$ for $\cW$. Following \cite{YEldar1,YEldar2} (see also  \cite{CE06}), given a Bessel sequence $\cG=\{g_i\}_{i\in \mathbb I}$  in $\cV$ we say that $\cG$ is a  (oblique) $\cV$-dual of $\cF$   if 
$$ 
f=\sum_{i\in \mathbb I} \langle f,g_i\rangle \ f_i \fe f\in \cW\ .
$$ 
It turns out (see \cite{YEldar1,YEldar2}) that if $\cG$ is a $\cV$-dual of $\cF$ then 
$T_\cF\,T_\cG^*=P_{\cW//\cV^\perp}$, where $P_{\cW//\cV^\perp}$ denotes the oblique projection with range $\cW$ and null space $\cV^\perp$. In this case, 
by taking adjoints in the previous identity we also get that $T_\cG\,T_\cF^*= P_{\cW//\cV^\perp}^*=P_{\cV//\cW^\perp}$.
Hence, $T_\cG$ is onto $\cV$ and then $\cG$ is a frame for $\cV$; moreover, we obtain the reconstruction formula
$$ 
g=\sum_{i\in \mathbb I} \langle g,f_i\rangle \ g_i \fe g\in \cV\ .
$$
We  consider the set of oblique $\cV$-duals of $\cF$ given by
\begin{equation}\label{eq duales}
\cD_\cV(\cF)=\left\{\cG=\{g_i\}_{i\in \mathbb I} \text{ is a }\cV\text{-dual of } \cF\right\}\ .
\end{equation}

\begin{rem}\label{rem classical canonical dual}
Let $\cF=\{f_i\}_{i\in \mathbb I} $ be a frame for $\cW$. If we set $\cV=\cW$ then a Bessel sequence $\cG$ in $\cW$ is a $\cW$-dual of $\cF$ if it is a dual frame for $\cF$ in $\cW$ in the classical sense (see \cite{Chrisbook}) i.e. $T_\cG\,T_\cF^*=P_\cW$. Hence
$$
\cD_\cW(\cF)=\cD(\cF)\igdef
\left\{\cG=\{g_i\}_{i\in \mathbb I} \text{ is a dual frame for } \cF \text{ in } \cW\,   \right\}\ .
$$ Recall that there is a distinguished (classical) dual, called the canonical dual of $\cF$, denoted $\cF^\#=\{f_i^\#\}_{i\in \mathbb I}$, given by $f_i^\#=S_\cF^\dagger f_i$ for $i\in \mathbb I$, where $S_\cF^\dagger$ denotes the Moore-Penrose pseudo-inverse of the (closed range positive semidefinite operator ) $S_\cF\,$.

\pausa
In the general context of oblique duality there also exists a distinguished $\cV$-dual for $\cF$, the so-called {\it canonical $\cV$-dual}, which we denote by $$
\cF^\#_{\cV}=\{f^\#_{\cV,\,i}\}_{i\in \mathbb I} \peso{given by} f^\#_{\cV,\,i}=P_{\cV//\cW^\perp} f^\#_i  \fe i\in \mathbb I  \ ,
$$ 
where $\cF^\#=\{f_i^\#\}_{i\in \mathbb I}$ is the canonical dual of $\cF$. 
It turns out that the encoding-decoding scheme based on the oblique dual pair $(\cF,\cF^\#_{\cV})$ has several optimality properties (see \cite{YEldar1,YEldar2}). 
\EOE
\end{rem}

\subsection{Shift-invariant subspaces and frames of translates}\label{SI cosas}

In what follows we consider $L^2(\R^k)$ (with respect to Lebesgue measure) as a separable and complex Hilbert space.
Recall that a closed subspace $\cV\subseteq L^2(\R^k)$ is {\it shift-invariant} (SI) if $f\in \cV$
implies $T_\ell f \in \cV$ for any $\ell\in \Z^k$, where $T_yf(x)=f(x-y)$ is
the translation by $y \in \R^k$. 
For example, if we take a subset $\cA \subset
L^2(\R^k)$ then, \beq \label{def: SI} \cS(\cA)= \overline{\text{span}}\,\ \{T_\ell f: f\in \cA,\  \ell \in
\Z^k\} \eeq
is a shift-invariant subspace called the {\it SI subspace generated by $\cA$}. Indeed, $\cS(\cA)$ is the smallest SI subspace that contains $\cA$. We say that a SI subspace $\cV$ is {\it finitely generated} (FSI) if there exists a finite set $\cA\subset L^2(\R^k)$ such that $\cV=S(\cA)$. In this case, the {\it length} of $\cV$ is the smallest cardinal $\#(\cA)$ such that $S(\cA)=\cV$.

\pausa In order to describe the fine structure of a SI subspace we consider the following representation of $L^2(\R^k)$ (see 
\cite{ BDR,Bo,Helson,RS95} and \cite{CabPat} for extensions of these notions to the more general context of actions of locally compact abelian groups). Let $\T=[-1/2,1/2)$ be endowed with the Lebesgue measure and let $L^2(\T^k, \ell^2(\Z^k))$ be the Hilbert space of square integrable $\ell^2(\Z^k)$-valued functions that consists of all vector valued measurable functions $\phi: \T^k \to \ell^2(\Z^k)$ with the norm 
$$\| \phi\|^2= \int_{\T^k} \| \phi(x)\|_{\ell^2(\Z^k)}^{2} \ dx< \infty.$$
Then, $\Gamma: L^2(\R^k)\to L^2(\T^k, \ell^2(\Z^k))$ defined for $f\in L^1(\R^k)\cap L^2(\R^k)$ by 
\beq\label{def: iso}
\Gamma f: \T^k \to \ell^2(\Z^k)\ ,\quad\Gamma f(x)= (\hat{f}(x+\ell))_{\ell\in \Z^k},
\eeq
extends uniquely to an isometric isomorphism between $L^2(\R^k)$ and $L^2(\T^k, \ell^2(\Z^k))$; here $\hat f$ denotes the Fourier transform of $f\in L^2(\R^k)$.

\pausa 
Let $\cV\subset L^2(\R^k)$ be a SI subspace. Then, there exists a function $J_\cV:\T^k\rightarrow\{$ closed subspaces of 
$\ell^2(\Z^k)\}$ such that: if $ P_{J_\cV(x)}$ denotes the orthogonal projection onto $J_\cV(x)$ for $x\in\T^k$, then for every $\xi,\,\eta\in \ell^2(\Z^k)$ the function $x\mapsto \langle P_{J_\cV(x)} \,\xi\coma \eta\rangle$ is measurable and 
\beq\label{pro: V y J}
\cV=\{ f\in L^2(\R^k): \Gamma f(x) \in J_\cV(x) \,\ \text{for a.e.}\,\ x\in \T^k\}.
\eeq The funcion $J_\cV$ is the so-called {\it measurable range function} associated with $\cV$. By \cite[Prop.1.5]{Bo}, Eq. \eqref{pro: V y J} establishes a bijection between 
SI subspaces of $L^2(\R^k)$ and measurable range functions.
In case 
$\cV=S(\mathcal A) \subseteq L^2(\R^k)$ is the SI subspace generated by $\mathcal A=\{h_i:i\in \mathbb I\}\subset L^2(\R^k)$, where $\mathbb I$ is a finite or countable infinite set, then for a.e. $x\in\T^k$ we have that 
\beq\label{eq Jv}
J_\cV(x)=\{\Gamma h_i(x): \ i\in \mathbb I\}^{-\|\cdot\|}\,.
\eeq

\pausa
%-------------------------------------
%
%Given SI subspaces $\cV$ and $\cW$, recall that a bounded linear transformation 
%$S\in L(\cV,\cW)$ is {\it shift preserving} (SP) if 
%$T_\ell \, S=S\,T_\ell$ for every $\ell\in\Z^k$. 
%In this case, for $x\in \T^k$, let 
%$ [S]_x\in L(J_\cV(x),J_\cW(x))$ be the linear transformation determined by 
%\beq\label{defi hatS2}
 %[S]_x \big(\Gamma f(x)\,\big)=\Gamma (Sf) (x) \quad \text{ for }\ f\in \cV\,.
%\eeq
%Then, (see \cite[Prop.4.5]{Bo}) the function $[S]:\T^k\rightarrow \coprod_{x\in \T^k}L(J_\cV(x),J_\cW(x))$ is weakly measurable, i.e. 
%for every $\xi,\,\eta\in \ell^2(\Z^k)$ the function 
%$x\mapsto \langle [S]_x\,P_{J_\cV(x)} \,\xi\coma \eta \rangle $ 
%is measurable, and essentially bounded; indeed 
%$\text{ess}\,\sup_{x\in \T^k}\|[S]_x\|=\|S\|$. 
%Conversely, if $s:\T^k\rightarrow \coprod_{x\in \T^k}L(J_\cV(x),J_\cW(x))$ is weakly measurable and essentially bounded then there exists a unique bounded linear transformation $S\in L(\cV,\cW)$ that is SP and such that $[S]=s$.
%
%---------------------------------------------
Recall that a bounded linear transformation 
$S\in L(L^2(\R^k))$ is {\it shift preserving} (SP) if 
$T_\ell \, S=S\,T_\ell$ for every $\ell\in\Z^k$. In this case (see \cite[Thm 4.5]{Bo}) there exists
a (weakly) measurable field of operators $[S]_{(\cdot)}:\T^k\rightarrow \ell^2(\Z^k)$ (i.e. such that 
for every $\xi,\,\eta\in \ell^2(\Z^k)$ the function 
$\T^k\ni x\mapsto \langle [S]_x\, \xi\coma \eta \rangle $ 
is measurable) and essentially bounded (i.e. the function $\T^k\ni x\mapsto \|\,[S]_{x}\,\|$ is essentially bounded)
such that 
\beq\label{defi hatS}
 [S]_x \big(\Gamma f(x)\,\big)=\Gamma (Sf) (x) \quad \text{ for a.e. }x\in\T^k\ , \ \ f\in L^2(\R^k)\,.
\eeq Moreover, $\|S\|=\esssup_{x\in\T^k} \|\, [S]_x\, \|$.
Conversely, if $s:\T^k\rightarrow L(\ell^2(\Z^k))$ is a weakly measurable and essentially bounded field of operators then,
 there exists a unique bounded operator $S\in L(L^2(\R^k))$ that is SP and such that $[S]=s$. 
For example, let $\cV$ be a SI subspace and consider $P_\cV\in L(L^2(\R^k))$, the orthogonal projection 
onto $\cV$; then, $P_\cV$ is SP so that 
$[P_\cV]{} :\T^k\rightarrow L(\ell^2(\Z^k))$ is given by $[P_\cV]{}_x=P_{J_\cV(x)}$ i.e., the orthogonal projection onto $J_\cV(x)$, for a.e. $x\in\T^k$.

\pausa The previous notions associated with SI subspaces and SP operators allow to develop a detailed study of frames of translates. 
Indeed, let $\cF=\{f_i\}_{i\in \mathbb I}$ be a (possibly finite) sequence in $L^2(\R^k)$. We denote by $E(\cF)$ the family of translates of $\cF$, namely $E(\cF)=\{T_\ell f_i\}_{(\ell,i)\in \Z^k\times \mathbb I}$. For $x\in \T^k$, let $\Gamma\cF(x)=\{\Gamma f_i(x)\}_{i\in \mathbb I}$ which is a (possibly finite) sequence in $\ell^2(\Z^k)$. Then \cite{Bo,RS95} $E(\cF)$ is a $b$-Bessel sequence if and only if $\Gamma\cF(x)$ is a $b$-Bessel sequence for a.e. $x\in \T^k$. In this case,  
we consider $T_{\Gamma\cF(x)}:\ell^2(\mathbb I)\rightarrow \ell^2(\Z^k)$ and $S_{\Gamma\cF(x)}:\ell^2(\Z^k)\rightarrow \ell^2(\Z^k)$  the  synthesis and frame operators of $\Gamma\cF(x)$, respectively, for $x\in\T^k$; it is straightforward to check that $S_{E(\cF)}$ is a SP operator. 

\pausa
If $\cF=\{f_i\}_{i\in \mathbb I}$ and $\cG=\{g_i\}_{i\in \mathbb I}$ are such that $E(\cF)$ and $E(\cG)$ are Bessel sequences then (see \cite{HG07,RS95}) the following fundamental relation holds: 
\beq\label{eq:fourier}
[T_{E(\cG)}\,T^*_{E(\cF)}]{}_x 
= T_{\Gamma\cG(x)}\,T^*_{\Gamma\cF(x)}\ , \quad  \text{for a.e }\, x \in \T^k \,.
\eeq
Eq. \eqref{eq:fourier} has several consequences; indeed, if 
$\cW$ is a SI subspace of $L^2(\R^k)$ and we assume further that $\cF,\,\cG\in\cW^n$ then:
\begin{enumerate}
\item For every $f,\,g\in L^2(\R^k)$, 
$$
\langle S_{E(\cF)}\, f,\,g\rangle =\int_{\T^k} \langle S_{\Gamma\cF(x)} 
\ \Gamma f(x),\,\Gamma g(x)\rangle_{\ell^2(\Z^k)}\ dx\ . 
$$ 
This last fact implies that $[S_{E(\cF)}] _x=S_{\Gamma \cF(x)}$ for a.e. $x\in\T^k$; moreover, it also holds that 
$E(\cF)$ is a frame for $\cW$ with frame bounds 
$0<a\leq b$ if and only if $\Gamma\cF(x)$ is a 
frame for $J_\cW(x)$ with frame bounds $0<a\leq b$ for a.e. $x\in \T^k$ (see \cite{Bo}).
\item Since $[P_\cW]_x = P_{J_\cW(x)}$ for a.e. $x\in\T^k$ then $E(\cG)$ is a (classical) dual for $E(\cF)$ in $\cW$ if and only if $\Gamma \cG(x)$ is a (classical) dual for $\Gamma \cF(x)$ in $J_\cW(x)$ for a.e. $x\in\T^k$ (see \cite{Bo,HG07,HG09}).
\end{enumerate}

\subsection{Majorization in probability spaces}\label{2.3}

Majorization between vectors (see \cite{Bhat,MaOl}) has played a key role in frame theory. On the one hand, majorization allows to characterize the existence of frames with prescribed properties (see \cite{AMRS,CFMPS,CasLeo}). On the other hand, majorization is a preorder relation that implies a family of tracial inequalities; this last fact can be used to explain the structure of minimizers of the so-called Benedetto-Fickus frame potential (\cite{BF,CKFT}) as well as more general convex potentials for finite frames (see \cite{MR08,MR10,MRS13,MRS14,MRS13b}). In the next section we extend the notion of convex potentials to the context of Bessel families of translates of finite sequences; therefore, we will need the following general notion of majorization between functions in probability spaces.

\pausa
Throughout this section the triple $(X,\cX,\mu)$ denotes 
a probability space i.e. 
 $\cX$ is a $\sigma$-algebra of sets in $X$ and $\mu$ is a probability measure defined on $\cX$. We shall denote by $L^\infty(X,\mu)^+ 
= \{f\in L^\infty(X,\mu): f\ge 0\}$. 
For $f\in L^\infty(X, \mu)^+$,
the {\it decreasing rearrangement} of $f$ (see \cite{MaOl}), denoted $f^*:[0,1)\rightarrow \R^+$, is given by
\beq\label{eq:reord}
f^*(s ) \igdef\sup \,\{ t\in \R^+ : \ \mu \{x\in X:\ f(x)>t\} >s\} 
\fe s\in [0,1)\, .
\eeq 

\begin{rem} \label{rem:prop rear elem}We mention some elementary facts related with the decreasing rearrangement of functions that we shall need in the sequel. Let $f\in L^\infty(X,\mu)^+$, then: 
\ben
\item $f^*$ is a right-continuous and non-increasing function. 
\item $f$ and $f^*$ are equimeasurable i.e. for every Borel set $A\subset \R$ then $\mu(f^{-1}(A))=|(f^*)^{-1}(A)|$, where $|B|$ denotes the Lebesgue measure of the Lebesgue measurable set $B\subset \R$. In turn, this implies that for every continuous $\varphi:\R^+\rightarrow \R^+$ then: $\varphi\circ f\in L^\infty(X,\mu)$ iff $\varphi\circ f^*\in L^\infty([0,1])$ and in this case $$ \int_X \varphi\circ f\ d\mu =\int_{0}^1 \varphi\circ f^*\ dx\,. $$
\item If $g\in L^\infty(X,\mu)$ is such that $f\leq g$ then $0\leq f^*\leq g^*$; moreover, in case $f^*=g^*$ then $f=g$. 
\item If we consider the probability space $([0,1],\mathcal B,\, dt)$ - Lebesgue measurable sets in [0,1] with Lebesgue measure - then $f^*\in L^\infty([0,1],\, dt)$ is such that $(f^*)^*=f^*$.
\item If $c \in \R$ is such that $f+c\geq 0$ then $(f+c)^*= f^*+c$.\EOE
\een
\end{rem}

\begin{fed}\rm Let $f, g\in L^\infty (X, \mu)^+$ and 
let $f^*,\,g^*$ denote their decreasing rearrangements. We say that $f$ \textit{submajorizes} $g$ (in $(X,\cX,\mu)$), denoted $g \prec_w f$, if
\begin{eqnarray*}\label{eq: mayo de func}
\int_{0}^{s} g^*(t) \,\ dt &\leq& \int_{0}^{s} f^*(t)\,\ dt \fe  0\leq s\leq 1 \,.
\end{eqnarray*}
If we further have that $\int_{0}^{1} g^*(t)\,\ dt = \int_{0}^{1} f^*(t)\,\ dt$ 
 we say that $f$ \textit{majorizes} $g$ and write $g \prec f$. \EOE
\end{fed}
\pausa
In order to check  majorization between functions in probability spaces, we can consider the so-called {\it doubly stochastic maps}. Recall that a linear operator $D$ acting on $L^\infty(X,\mu)$ is a doubly-stochastic map if $D$ is unital, positive and trace preserving i.e. 
$$ 
D(1_X)=1_X \ , \ \ 
D\big(\, L^\infty (X, \mu)^+ \, \big)\inc L^\infty (X, \mu)^+ 
\py 
\int_X D(f)(x)\ d\mu(x) =\int_X f(x)\ d\mu(x)  
$$ 
for every $f\in L^\infty(X,\mu)$. 
It is worth pointing out that $D$ is necessarily a contractive map.

\pausa
Our interest in majorization relies in its relation with integral inequalities in terms of convex functions. The following result summarizes this relation as well as the role of the doubly stochastic maps (see for example \cite{Chong,Ryff}).

\begin{teo}\label{teo porque mayo} \rm
Let $f,\,g\in  L^\infty (X, \mu)^+$. Then the following conditions are equivalent:
\ben
\item $g\prec f$;
\item There is a doubly stochastic map $D$ acting on $L^\infty(X,\mu)$ such that $D(f)=g$;
\item For every convex function $\varphi:\R^+\rightarrow \R^+$ we have that 
\beq\label{eq teo:desi mayo}
 \int_X \varphi(g(x)) \ d\mu(x)\leq  \int_X \varphi(f(x))\ d\mu(x)\ .
\eeq
\een 
In case we only have $g\prec_w f$ then Eq. \eqref{eq teo:desi mayo} holds if we assume further  that $\varphi$ is an increasing convex function. \qed
\end{teo}

\begin{exa}\label{exa: integral}
The operator $D$ given by $D(f)=(\int_X f\ d\mu)\cdot 1_X$ is a doubly stochastic map. Hence, we get the majorization relation
$(\int_X f\ d\mu)\cdot 1_X\prec f$. Therefore, if $\varphi:\R^+\rightarrow \R^+$ is any convex function and $f\in  L^\infty (X, \mu)^+$  then, by Theorem \ref{teo porque mayo}, we have that 
\beq \label{ecua jensen}
 \varphi(\int_X f\ d\mu)= \int_X \varphi((\int_X f\ d\mu)\cdot 1_X(x)) \ d\mu(x)\leq \int_X  \varphi(f(x))\ d\mu(x) \ ,
 \eeq
which is an instance of the classical Jensen's inequality. Using the previous facts, notice that if $c\in \R$ is such that $0\leq c\leq \int_X f\ d\mu$ then it is easy to see that $c\cdot 1_X\prec_w f$. \EOE
\end{exa}

\pausa
The following result will play a key role in the study of the structure of minimizers of $\prec_w$ within (appropriate) sets of functions.

\begin{pro}[\cite{Chong}]\label{pro int y reo} \rm
Let $f,\,g\in L^\infty(X,\mu)^+$ such that $g\prec_w f$. If there 
exists a non-decreasing and strictly convex function $\varphi:\R^+\rightarrow \R^+$ such that 
\beq 
\int_X \varphi(f(x))\ d\mu(x) =\int_X \varphi(g(x)) \ d\mu(x) 
\peso{then} g^*=f^* \ . \QEDP
\eeq 
\end{pro}
\pausa
With the notations of Example \ref{exa: integral} notice that Proposition \ref{pro int y reo} implies (the known fact) that if $\varphi$ is strictly convex and such that equality holds in Jensen's inequality Eq. \eqref{ecua jensen} then $f^*=\int_X f\ d\mu$ and hence $f=\int_X f\ d\mu$.

\section{Convex potentials for sequences of translates in FSI spaces}\label{sec3}

We begin by describing the convex potentials for finite sequences of vectors with respect to a finite dimensional subspace.
We consider the sets 
$$
\convf = \{ 
\varphi:\R^+\rightarrow \R^+\ , \  \varphi   \ \mbox{ is a convex function}  \ \} 
$$  
and $\convfs = \{\varphi\in \convf \ , \  \varphi$ is strictly convex $\}$. 

\pausa
Now, given $\varphi\in \convf$ and a finite dimensional subspace $\cW\subset \hil$, then the {\it convex potential} associated to $(\varphi,\cW)$, denoted by $P_\varphi^\cW$, is defined as follows:
for a finite sequence $\cF=\{f_i\}_{i\in \IN{n}}\in  \cW^n$ with  frame operator $S_\cF\in L(\hil)^+$, 
\beq\label{obs mejr}
P_\varphi^\cW(\cF)= \tr [\varphi(S_\cF)\, P_\cW]
\eeq
where $\varphi(S_\cF)\in L(\hil)^+$ is obtained by functional calculus 
and $\tr(\cdot)$ denotes the usual (semi-finite) trace in $L(\hil)$.
Notice that by construction, $P_\cW\,S_\cF=S_\cF\,P_\cW=S_\cF$: then, it is clear that 
\beq\label{eq defi pot fin seq}
P_\varphi^\cW(\cF)=\sum_{i\in\IN{d}}\varphi(\lambda_i(S_\cF))\, ,
\eeq where $d=\dim \cW$ and $(\lambda_i(S_\cF))_{i\in\IN{d}}\in (\R^+)^d$ denotes the vector of eigenvalues of the positive operator $S_\cF|_{\cW}\in L(\cW)^+$, counting multiplicities and arranged in non-increasing order (we use the convention $\IN{0}=\emptyset$). 
In particular, if $\varphi\in\convf$ is such that $\varphi(0)=0$ we get that $$P_\varphi^\cW(\cF)= 
\tr \, (\varphi(S_\cF))=\tr \, (\varphi(G_\cF)) \, ,$$  
where the $n\times n$ matrix $G_\cF=(\langle f_i,\,f_j\rangle)_{i,j\in\In}$ is the Gramian matrix of the finite sequence $\cF$. That is, if $\varphi(0)=0$, then $P^\cW_\varphi=P_\varphi$ does not depend on $\cW$. For example,
in case $\varphi(x)=x^2$, then $P_\varphi^\cW(\cF)=P_\varphi(\cF)$ coincides with the frame potential: indeed, by Eq. \eqref{rep BFpot} we have that 
\beq\label{eq: BF pot prep}
P^\cW_\varphi(\cF)=P_\varphi(\cF)=\tr(S_\cF^2)=\tr(G_\cF^2)
=\sum_{i,\,j\,\in \In}|\api f_i\coma f_j \cpi |\,^2=\FP(\cF)\ .
\eeq
For $\varphi\in \convf$ and a finite dimensional subspace $\cW\subset \hil$, 
$P_\varphi^\cW(\cF)$ is a measure of the spread of the eigenvalues of the frame operator of $\cF=\{f_i\}_{i\in\In}\in\cW^n$. That is, (under suitable normalization hypothesis on $\cF$) the smaller the value $P_\varphi^\cW(\cF)$ is, the more concentrated the eigenvalues of $S_\cF|_\cW\in L(\cW)^+$ are. This is the main motivation for considering these convex potentials (see \cite{MR10,MRS13,MRS13b,Pot}).

\pausa
Next we extend the notion of convex potential to the context of finitely generated shift invariant systems in FSI subspaces.

\begin{fed} \label{defi pot}\rm Let $\cW$ be a FSI subspace in $L^2(\R^k)$, let $\cF=\{f_i\}_{i\in\In}\in\cW^n$ 
be such that $E(\cF)$ is a Bessel sequence and consider $\varphi\in \convf$.
Then the convex potential associated to $(\varphi,\cW)$ on $E(\cF)$, denoted $P_\varphi^\cW(E(\cF))$, is given by
\beq\label{eq defi pot}
P_\varphi^\cW(E(\cF))=\int_{\T^k} 
P_\varphi^{J_\cW(x)}(\Gamma\cF(x)) \ dx\, 
\eeq 
where  
$P_\varphi^{J_\cW(x)}(\Gamma \cF(x))
=\tr(\varphi(S_{\Gamma \cF(x)})\, [P_\cW]_x)$  
is the convex potential associated with $(\varphi,J_\cW(x))$ of the sequence $\Gamma\cF(x)=\{\Gamma \, f_i(x)\}_{i\in\In}$ in $\ell^2(\Z^k)$, for every 
$x\in \T^k$. 
\EOE
\end{fed}

\pausa
Next we develop some notions and tools in order to show that the right hand side in Eq. \eqref{eq defi pot} is well defined, namely that the function $\T^k\ni x\mapsto P_\varphi^{J_\cW(x)}(\Gamma\cF(x)) $ is integrable.

\pausa
Let $\cF=\{f_i\}_{i\in\In}$ be a finite sequence in $L^2(\R^k)$ such that $E(\cF)$
is a Bessel sequence. Recall that in this case $S_{E(\cF)}$ is a SP operator and that for a.e. $x\in\T^k$, $[S_{E(\cF)}] _x= S_{\Gamma\cF(x)}\in L(\ell^2(\Z^k))^+$ is a positive and finite rank operator.

\pausa
The next lemma is a reformulation of a result in \cite{RS95} concerning the existence of measurable functions of eigenvalues and eigenvectors of measurable fields of positive semidefinite $n\times n$ matrices. 
\begin{lem}\label{lem spect represent ese}
Let $\cW$ be a FSI subspace in $L^2(\R^k)$ and let  $\cF=\{f_i\}_{i\in\In}\in\cW^n$ be such that $E(\cF)$
is a Bessel sequence.  Then, there exist: 
\begin{enumerate}
\item a measurable function $r:\T^k\rightarrow \N_{\geq 0}$ and measurable vector fields 
$v_j:\T^k\rightarrow \ell^2(\Z^k)$ for $j\in\In$ such that $r(x)\leq n$ and $\{v_j(x)\}_{j=1}^{r(x)}$ is an orthonormal system in $J_\cW(x)$ for a.e. $x\in\T^k$;
\item\label{lema1} bounded measurable functions $\la_j:\T^k\rightarrow \R^+$ for $j\in\In$, such that 
$\la_1 \ge \ldots \ge \la_n\,$, $\la_j(x)=0$ if $j>r(x)$ and 
\beq\label{lem repre espec S}
[S_{E(\cF)}]_x 
=\sum_{j=1}^{r(x)}\lambda_j(x)\ v_j(x)\otimes v_j(x) 
\ , \quad \text{ for a.e. } x\in\T^k\,.
\eeq
\end{enumerate}
If we assume further that $E(\cF)$ is a frame for $\cW$ then $r(x)=\dim J_{\cW}(x)$ and $\{v_j(x)\}_{j=1}^{r(x)}$ is an orthonormal basis (ONB) for $J_{\cW}(x)$ for a.e. $x\in\T^k$.
\end{lem}

\begin{proof}
Consider the measurable field of positive semidefinite matrices $G:\T^k\rightarrow \matn^+$ given by the Gramian 
$G(x)=(\langle \Gamma f_i(x),\Gamma f_j(x)\rangle)_{i,j\in\In}$, for $x\in \T^k$. Notice that $G(x)$ is the matrix
representation of $T_{\Gamma\cF(x)}^*T_{\Gamma\cF(x)}\in L(\C^n)^+$ with respect to the canonical basis of $\C^n$ for $x\in \T^k$. In particular, if $b$ denotes a Bessel (upper) bound of $E(\cF)$ then 
$$ \text{ ess sup}_{x\in\T^k} \|G(x)\|=\text{ ess sup}_{x\in\T^k} \| T_{\Gamma\cF(x)}T_{\Gamma\cF(x)}^*\|=\|S_{E(\cF)}\|\leq b\ ,$$ by the remarks at the end of Section \ref{SI cosas}. We set $r(x)=\rk(G(x))=\rk(S_{\Gamma \cF(x)})$ for $x\in\T^k$; therefore $r(\cdot):\T^k\rightarrow \N_{\geq 0}$ is a measurable function such that $r(x)\leq n$ for $x\in\T^k$. Hence, by considering a convenient finite partition of $\T^k$ into measurable sets we can assume, without loss of generality, that $r(x)=r\in \N$ for a.e. $x\in \T^k\,$. 

\pausa
Using results from \cite{RS95}, we see that there exist measurable functions $\la_j:\T^k\rightarrow \R^+$  
and measurable vector fields $u_j:\T^k\rightarrow \C^n$, for $j\in\In$, 
 such that: $\la_j(x)\geq \la_{j+1}(x)$ for $j\in\IN{n-1}$, $\{u_j(x)\}_{j\in\In}$ is an ONB of $\C^n$ and $G(x)u_j(x)=\la_j(x)\,u_j(x)$ for $j\in\In$ and a.e. $x\in \T^k$. 
In particular, the functions $\la_j:\T^k\rightarrow \R^+$ satisfy $0\leq \la_j(x)\leq \|G(x)\|\leq b$ for a.e. $x\in\T^k$, $j\in\In$; these remarks prove item \ref{lema1} above.

\pausa
Take the polar decomposition $T_{\Gamma\cF(x)}=U(x)\,|T_{\Gamma\cF(x)}|$, where $U(x):\C^n\rightarrow J_{\cW}(x)\subset \ell^2(\Z^k)$ is (the unique) partial isometry with $\ker U(x)=\ker T_{\Gamma\cF(x)}$ for a.e. $x\in \T^k$.  Hence, in this case $U(x)=T_{\Gamma\cF(x)}\,(G^{1/2}(x))^\dagger$ and therefore $U(\cdot):\T^k\rightarrow L(\C^n,\ell^2(\Z^k))$ is a well defined measurable field of partial isometries.
Then, $v_j:\T^k\rightarrow \ell^2(\Z^k)$ given by $v_j(x)=U(x)\,u_j(x)\in J_{\cW}(x)$ for $j\in\IN{n}$ and $x\in\T^k$ are measurable vector fields 
such that $\{v_j(x)\}_{j\in\IN{r}}$ is an orthonormal system in $J_{\cW}(x)$, for a.e. $x\in\T^k$; moreover, $[S_{E(\cF)}]_x \,v_j(x)=\la_j(x)
\,v_j(x)$ for $j\in\IN{r}$ and a.e. $x\in\T^k$. 
Since $\rk[S_{E(\cF)}]_x=r$ for a.e. $x\in\T^k$, then we see that Eq. \eqref{lem repre espec S} holds in this case.

\pausa
Finally, notice that if $E(\cF)$ is a frame for $\cW$ then we should have that $r=\rk [S_{E(\cF)}]_x=\dim J_{\cW}(x)$ for a.e. $x\in\T^k$ which shows the last part of the statement.
\end{proof}

\begin{rem}\label{fine la} \rm
Let $\cF=\{f_i\}_{i\in\In}$ be a finite sequence in $L^2(\R^k)$ such that $E(\cF)$ 
is a Bessel sequence. By Lemma \ref{lem spect represent ese} there exist measurable 
vectors fields $v_j:\T^k\rightarrow \ell^2(\Z^k)$ and measurable functions $\la_j:\T^k\rightarrow \R^+$ such that
they verify Eq. \eqref{lem repre espec S}. In what follows we consider the {\it fine spectral structure} of 
$E(\cF)$ that is the weakly measurable function
\beq\label{el lambda}
%\la(S_{\Gamma \cF(x)}) \igdef 
\T^k\ni x\mapsto (\lambda_j([S_{E(\cF)}]_x)\,)_{j\in\N}\in \big(\,\ell^1_+(\Z^k)\,\big)\da \  \ \text{(non-increasing sequences) } ,
\eeq 
where $\lambda_j([S_{E(\cF)}]_x)=\la_j(x)$ for $j\in \I_{r(x)}$ and $\lambda_j([S_{E(\cF)}]_x=0$ for $j\geq r(x)+1$, for $x\in \T^k$.
Hence, $(\lambda_j([S_{E(\cF)}]_x)\,)_{j\in\N}$ coincides with the  
sequence  of eigenvalues of the positive semidefinite finite rank operator 
$[S_{E(\cF)}]_x=S_{\Gamma \cF(x)}\in L(\ell^2(\Z^k))$, counting multiplicities and arranged in non-increasing order, for a.e. $x\in\T^k$.
%Notice that the maps $\la_j(x)$ of Lemma \ref{lem spect represent ese} are uniquely determined (a.e.) by the equation 
%$\la_j(x)=\lambda_j(S_{\Gamma \cF(x)})$,  for $j\in\In\,$. The last entries are 
%$\lambda_j(S_{\Gamma \cF(x)})=0$ for $j\geq n+1$ and $x\in \T^k$. Therefore, another way to state item 2 of 
%Lemma \ref{lem spect represent ese} is to say that the map $\T^k\ni x\mapsto \la(S_{\Gamma \cF(x)})$ of Eq. \eqref {el lambda} 
%is measurable. 
%We refer to this function as the {\it fine spectral structure} of $E(\cF)$. 
\EOE\end{rem}

\begin{rem}\label{ahora chi} 
Consider the notations from Definition \ref{defi pot}. We now show that the right hand side in Eq. \eqref{eq defi pot} is well defined.  
Indeed, by Lemma \ref{lem spect represent ese} we get a spectral representation of $[S_{E(\cF)}]_{(\cdot)}$ as in Eq. \eqref{lem repre espec S} in terms of the bounded and measurable functions $\la_j(\cdot):\T^k\rightarrow \R^+$, for $j\in\In$. If we consider the bounded and measurable function $d(x)=\dim J_\cW(x)\geq r(x)$ for $x\in\T^k$ then, using Eq. \eqref{eq defi pot fin seq} we see that 
$$ P_\varphi^{J_\cW(x)}(\Gamma\cF(x)) = \sum_{j\in\IN{r(x)}} \varphi(\la_j(x))+(d(x)-r(x))\ \varphi(0)\quad \text{for a.e. } x\in\T^k\ .$$
Hence, the non-negative function $$ \T^k\ni x\mapsto P_\varphi^{J_\cW(x)}(\Gamma\cF(x))$$ is bounded and measurable and therefore integrable on $\T^k$. This shows that the convex potential $P_\varphi^\cW(E(\cF))$ is a well defined non-negative real number. \EOE
\end{rem}

\pausa
Incidentally, Remark \ref{ahora chi} above shows that if $\varphi(0)=0$ then the convex potential $P^\cW_\varphi=P_\varphi$ does not depend on the FSI subspace $\cW$.

\begin{exa}
Let $\cW$ be a FSI subspace of $L^2(\R^k)$ and let $\cF=\{f_i\}_{i\in\In}\in\cW^n$. If we set $\varphi(x)=x^2$ for $x\in \R^+$ then, the corresponding potential on $E(\cF)$, that we shall denote $\FP(E(\cF))$, is given by
$$ 
\FP(E(\cF))= \int_{\T^k}
\tr( S_{\Gamma \cF(x)}^2) \ dx= \int_{\T^k} \ \sum_{i,\,j\in \In} |\langle \Gamma f_i(x),\Gamma f_j(x)\rangle|^2 \ dx\ .$$
Hence,   $\FP(E(\cF))$ is a natural extension of the Benedetto-Fickus frame potential of Eq. \eqref{eq: BF pot prep}.
\EOE
\end{exa}

\begin{rem}\label{local trace}
Let $\cW$ be a SI subspace of $L^2(\R^k)$ and let $A\in L(\ell^2(\Z^k))^+$ be a positive operator: 
in \cite{Dutk}, E. Dutkay introduces the local trace function of $A$ relative to $\cW$, denoted $\tau_{\cW,\,A}:\T^k\rightarrow [0,\infty]$ as follows:  for $x\in\T^k$, $$
\tau_{\cW,\,A}(x)=\tr(A \, [P_\cW]_x)\ ,  
$$
where $\tr(\cdot)$ denotes the usual (semi-finite) trace in $L(\ell^2(\Z^k))$.
We can extend the notion of local trace function as described above to the following setting: given $T\in L(L^2(\R^k))^+$ a positive and SP operator, we let the local trace function of $T$ with respect to the SI subspace $\cW$ be given by 
\beq\label{def: loc tr} \tau_{\cW,\,T}(x)
=\tr([ T]_x \, [P_\cW]_x)\ ,\quad x\in\T^k\, . 
\eeq Notice that if $A\in L(\ell^2(\Z^k))^+$ and $T\in L(L^2(\R^k))^+$ is the unique positive and SP operator such that $[ T]_x=A$ for $x\in\T^k$ then 
$$\tau_{\cW,\,A}(x)= \tau_{\cW,\,T}(x) \ , \quad x\in \T^k\, .$$
If we assume further that  $\cW$ is a FSI subspace, we consider $\varphi\in\convf$ and take $\cF=\{f_i\}_{i\in\In}\in\cW^n$ then
$$
P_\varphi^\cW(E(\cF))=\int_{\T^k}
\tau_{\cW,\varphi(S_{E(\cF)})}(x)\ dx\, , $$
where $\varphi(S_{E(\cF)})\in L(L^2(\R^k))^+$ is obtained by the functional calculus. Indeed, notice that in this case $\varphi(S_{E(\cF)})$ is a SP operator such that 
\beq
[\varphi(S_{E(\cF)})]_x 
=\varphi( \,[S_{E(\cF)}]_x)=\varphi(S_{\Gamma \cF(x)})\ , \quad \text{ for a.e. } x\in\T^k\  . \EOEP \eeq
\end{rem}

\pausa
Let $\cW$ be a FSI subspace. In what follows we show that, under some natural restrictions, the convex potentials $P_\varphi^\cW(E(\cF))$ for finite sequences $\cF\in\cW^n$ detect tight frames for $\cW$ as their minimizers (see Theorem \ref{pro min del PF} below).  
In turn, this last fact motivates the study of the structure of minimizers of convex potentials for finitely generated sequences in $L^2(\R^k)$ (under some restrictions) since these minimizers can be considered as natural substitutes of tight  frames. In order to state the results on this matter, we introduce the following notions and notations.

\begin{rem}\label{rem const de esp}
Let $(X, \mathcal X, \mu_X), (Y, \mathcal Y, \mu_Y)$ be two measure spaces; we consider their direct sum, denoted $X\bigoplus Y$, given by the three-tuple $(X\oplus Y, \mathcal X \bigoplus \mathcal Y, \mu_X \oplus \mu_Y)$, where
\ben
\item $X\oplus Y = X \stackrel{d}{\cup}Y$ (the disjoint union of the sets); we further consider the canonical inclusions $\eta_X:X\rightarrow X\oplus Y$ and $\eta_Y:Y\rightarrow X\oplus Y$ of $X$ and $Y$ into their disjoint union; hence $\eta_X$ and $\eta_Y$ are injective functions such that $\eta_X(X)\cap\eta_Y(Y)=\emptyset$ and $\eta_X(X)\cup\eta_Y(Y)=X\oplus Y$.
\item $\mathcal X \bigoplus \mathcal Y =\{A\oplus B=\eta_X(A)\cup \eta_Y(B): A\in \mathcal X, \, B\in \mathcal Y\}$;
\item $\mu_X \oplus \mu_Y$ is the measure given by $\mu_X \oplus \mu_Y(A\oplus B)= \mu_X (A)+ \mu_Y(B)$;
\een
 Notice that using the maps $\eta_X$ and $\eta_Y$ we can consider (as we sometimes do) $X,\, Y\subset X\oplus Y$ .
\EOE
\end{rem}

\begin{nota}\label{nota impor} In what follows we consider:
\ben
\item  A FSI subspace of $L^2(\R^k)$ of length $\ell$, denoted $\cW$;
\item $\cF=\{f_i\}_{i\in \In} \in \cW^n$ such that $E(\cF)$ is a Bessel sequence;
\item $d(x)=\dim J_\cW(x)\leq \ell$, for $x \in \T^k$;
\item The Lebesgue measure on $\R^k$; denoted $|\cdot|$ ; $X_i=d^{-1}(i)\subseteq \T^k$ and $p_i=|X_i|$, $i\in\IN{\ell}\,$.
\item We denote by $\const = \sum_{i\in\IN{\ell}} i\cdot p_i\,$.
\item The spectrum of $\cW$ is the measurable set $\text{Spec}(\cW) =
\bigcup_{i\in\IN{\ell}} X_i = \{x\in \T^k: d(x)\neq 0\}$. 
\EOE \een 
\end{nota}

\begin{teo}[Structure of $P_\varphi^\cW$ minimizers with norm restrictions]\label{pro min del PF}
Consider the Notations \ref{nota impor} and assume that 
$\sum_{i\in\In}\|f_i\|^2=1$. If $\varphi\in \convf$, then 
\beq\label{eq min de PF}
P_\varphi^\cW(E(\cF))\geq 
\const\ \varphi ( \const^{-1})  \, . 
\eeq
Moreover, if $\varphi\in \convfs$ then equality holds in \eqref{eq min de PF} iff $E(\cF)$ is a tight frame for $\cW$ i.e.
\beq\label{eq el min de PF}
S_{E(\cF)} = \const^{-1} \ P_{\cW}\, .
\eeq
\end{teo}

\begin{proof} Let $(X_{ij}, \cX_{ij},\, |\cdot|_{ij})$ where $X_{ij}=X_i\coma
\cX_{ij}=\cX_i$ the $\sigma$-algebra of Lebesgue measurable sets in $X_i$ and $|\cdot|_{ij}=|\cdot|_i$ the Lebesgue measure in $X_i$, for $j\in \I_i$ and $i\in \IN{\ell}$.
 With the notations in Remark \ref{rem const de esp}, we consider the measure space 
$$
(X, \mathcal X, \mu)= \bigoplus_{i\in \IN{\ell}} \bigoplus_{j\in \I_i}(X_{ij}, \mathcal X_{ij}, |\cdot|_{ij}) \ .
$$
For $i\in\IN{\ell}$ and $j\in\IN{i}$ we further consider the canonical inclusions $\eta_{i,j}:X_{i,j}\rightarrow X$.
Hence, for every $x\in X$ there exists unique $i\in\IN{\ell}$, $j\in\IN{i}$ and $\tilde x\in X_{i,j}=X_i$ such that $\eta_{i,j}(\tilde x)=x$.
Notice that by construction, $\mu(X)=\sum_{i\in\IN{\ell}} i\cdot p_i= \const\,$. 

\pausa
Let $\la_{E(\cF)}: X\to \R^+$ be the measurable function of eigenvalues of $E(\cF)$ defined as follows: 
for $x\in X$, let $(i\coma j)\in  \I_\ell\times \I_i\,$ and $\tilde x\in X_{i,j}=X_i$ be (uniquely determined) 
such that $\eta_{i,j}(\tilde x)=x$; in this case we set
$$
\la_{E(\cF)}(x)=\la_j (\, [S_{E(\cF)}]_{\tilde x})=\la_j(S_{\Gamma\cF(\tilde x)})\ ,
$$ 
where $\T^k\ni x\mapsto (\lambda_j([S_{E(\cF)}]_x)\,)_{j\in\N}\in \big(\,\ell^1_+(\Z^k)\,\big)\da $
%$\la(S_{\Gamma\cF(\tilde x)})\in \big(\,\ell^1_+(\Z^k)\,\big)\da $ 
is the fine spectral structure of $E(\cF)$ defined in Remark \ref{fine la}.
%Eq. \eqref{el lambda}.
We claim that if $\varphi \in \convf$, then \beq\label{eq ident int}
P_\varphi^\cW (E(\cF))=\int_{X} \varphi(\la_{E(\cF)}(x))\  d\mu(x)\,.
\eeq
Indeed, for Eq. \eqref{eq defi pot} 
$$
P_\varphi^\cW(E(\cF))= 
\int_{\T^k} P_\varphi^{J_\cW(x)}(\Gamma\cF(x)) \ dx
=\int_{\text{Spec}(\cW)} P_\varphi^{J_\cW(x)}(\Gamma \cF(x))\ dx \ ,
$$ 
where $P_\varphi^{J_\cW(x)}(\Gamma \cF(x))$ is the convex potential associated with $(\varphi,J_\cW(x))$ of the finite sequence $\Gamma\cF(x)=\{\Gamma \, f_i(x)\}_{i\in\In}$ in $\ell^2(\Z^k)$ as defined in Eq. 
\eqref{eq defi pot fin seq} (notice that  $P_\varphi^{J_\cW(x)}(\Gamma \cF(x))=0$ for $x\in\T^k\setminus \text{Spec}(\cW)$). Therefore, if $x\in X_i$ for some $i\in \IN{\ell}$ then
$$P_\varphi^{J_\cW(x)}(\Gamma \cF(x))=
\sum_{j=1}^i \varphi(\la_j(S_{\Gamma \cF(x)}))\,.$$
For $i\in\IN{\ell}$ we have that
$$
\int_{X_i} P_\varphi^{J_\cW(x)}(\Gamma \cF(x))\, dx = \int_{X_i} \ 
\sum_{j=1}^i \varphi(\la_j(S_{\Gamma \cF(x)}))
\, dx= \int_{\oplus_{j=1}^i X_{ij}} \varphi(\la_{E(\cF)}(x))\, d\mu(x) \ .
$$
Therefore, since Spec$(\cW)=\bigcup\limits_{i\in\IN{\ell}}X_i\,$ and $X=\oplus_{i\in\IN{\ell}}\oplus_{j\in\IN{i}}X_{i,j}\,$,  
$$
P_\varphi^\cW (E(\cF))=\sum_{i\in\IN{\ell}} \,
\int_{X_i} P_\varphi^{J_\cW(x)}(\Gamma \cF(x)) \, dx = \sum_{i\in\IN{\ell}}  \,\int_{\oplus_{j=1}^i X_{ij}} \varphi(\la_{E(\cF)}(x))\, d\mu(x)
= \int_{X} \varphi(\la_{E(\cF)}(x))\, d\mu(x)\, , 
$$ 
which proves Eq. \eqref{eq ident int}. In particular, if we take $\varphi(x)=x$ in Eq. \eqref{eq ident int}
we get that 
$$ 
\int_{X}\la_{E(\cF)}(x) \, d\mu (x)=\int_{\T^k} \tr (S_{\Gamma\cF(x)})\, dx 
= \int_{\T^k}\ \sum_{i\in \In} \|\Gamma f_i(x)\|^2\, dx=\sum_{i\in \In} \|f_i\|^2=1\ . 
$$
Consider the probability measure $\tilde{\mu} = \const\inv \, \mu$. 
Then, as in  Example \ref{exa: integral}, we have that 
\beq\label{eq: rel mayo cons}\int_{X}\la_{E(\cF)}(x)\, d\tilde{\mu}(x)=\const^{-1} \ \implies \ \const^{-1}\cdot 1_X\prec \la_{E(\cF)} 
\quad (\text{ in } (X, \mathcal X, \tilde{\mu})\, )\,.\eeq
If we let $\varphi\in \convf$ then, using the previous facts and Theorem 
\ref{teo porque mayo}, we get that
$$
\barr{rl}
\varphi\left(\const^{-1}\right) = 
 \int_{X}\varphi\left(\const^{-1}\cdot 1_X\right)\ d\tilde{\mu} & 
\stackrel{\ref{teo porque mayo}}{\le}   \int_{X}\varphi(\la_{E(\cF)}(x))\ d\tilde{\mu}(x) \\&\\
&= \const^{-1} \int_{X}\varphi(\la_{E(\cF)}(x))\ d{\mu}(x) 
\stackrel{\eqref{eq ident int}}{=} 
\const\inv \, P_\varphi^\cW(E(\cF)) \ , \earr
$$
which proves Eq. \eqref{eq min de PF}.
If $\varphi\in \convfs$ and also $P_\varphi^\cW(E(\cF))
=\varphi(\const^{-1}) \, \const\,$, using Eq. \eqref{eq ident int}, we get that
$$\int_{X}\varphi(\la_{E(\cF)}(x))\ d{\tilde \mu}(x)=\int_{X}\varphi(\const^{-1})\  d{\tilde \mu}\,.$$
Hence, by Proposition \ref{pro int y reo} and the majorization relation in Eq. \eqref{eq: rel mayo cons},
$$ (\la_{E(\cF)})^*= \const^{-1} \ 1_{[0,1]} \ \implies \ \la_i(\,[S_{E(\cF)}]_x)
= \const^{-1} \ \text{ for } i\in\IN{d(x)} \text{ and  a.e. } x\in\T^k\, .$$
 Therefore,
$S_{E(\cF)}=\const^{-1}\ P_{\cW}$ i.e. $E(\cF)$ is a tight frame for $\cW$. Conversely, notice that if $S_{E(\cF)} = \const^{-1} \ P_{\cW}\,$ then lower bound in Eq. \eqref{eq min de PF} is attained.
\end{proof}

%\begin{rem} \rm
%Notice that Theorem \ref{pro min del PF} above indicates that minimizers of the convex potentials have a nice {\it global} structure. That is, the structure of the minimizers of the convex potentials is not obtained by glueing 
%minimizers of the corresponding (local) convex potential in each fiber.
%\EOE
%\end{rem}

\section{Fine spectral structure of shift generated 
oblique duals in FSI subspaces}\label{sec4}

Throughout this section $\cV,\,\cW\subseteq L^2(\R^k)$ denote FSI subspaces such that $\cV\oplus \cW^\perp=L^2(\R^k)$ and $\cF=\{f_i\}_{i\in\In}\in\cW^n$ denotes a finite sequence such that $E(\cF)$ is a frame for $\cW$.

\pausa
Next we recall some characterizations of the condition $\cS\oplus \cT^\perp=L^2(\R^k)$ for SI 
subspaces and a characterization of shift generated (SG) oblique duals of $E(\cF)$; these results together 
with \cite{BMS14} allow us to obtain the exact value of the aliasing norm corresponding to 
the consistent sampling induced by the FSI subspaces $\cV$ and $\cW$. 
In Section \ref{sec 4.2} we obtain a detailed description of the fine spectral structure (i.e. eigenvalues) 
of the frame operators of SG oblique $\cV$-duals of the (fixed) frame $E(\cF)$ for $\cW$. 
We will apply these results in Section \ref{sec applic}, where we compute SG oblique dual frames with norm restrictions that simultaneously minimize the convex potentials $P_\varphi^\cV$ for all $\varphi\in\convf$.

\subsection{SG oblique duals and aliasing in FSI subspaces}

Following \cite{HG07} (see also \cite{YEldar3,HKK,HG09}) we consider the set of SG $\cV$-duals of $E(\cF)$:
\beq\label{defi: SG dual}
\duv = \cD_\cV^{SG}(E(\cF))=\{E(\cG)\in \cD_\cV(E(\cF))\, :\ \cG=\{g_i\}_{i\in\In}\in \cV^n \}\, .
\eeq 
In case $\cV=\cW$ then we write $\cD^{SG}(\cF)=\cD^{SG}_{\cW}(E(\cF))$ (which is the class of SG duals of type I, in the terminology of \cite{HG07}). Given $E(\cG)\in \duv$ we obtain the following (structured) reconstruction formulas: 
for every $f\in\cW$ and $g\in\cV$, 
$$ 
f=\sum_{(\ell,\,i)\in \Z^k\times \In} \langle f,\,T_\ell \,g_i\rangle \ T_\ell \,f_i \quad \py \quad 
g=\sum_{(\ell,\,i)\in \Z^k\times \In} \langle g,\,T_\ell\, f_i\rangle \ T_\ell \,g_i\, . 
$$

\pausa 
Next we describe some results related with the general assumption for studying oblique duality, namely $\cV\oplus \cW^{\perp}=L^2(\R^k)$, for the FSI subspaces $\cV$ and $\cW$, as well as SG oblique duality.
The next two results can be derived using combinations of results and techniques in \cite{AC09,KLL05,KLL06}.

\begin{lem}\label{lem: gen bow proj ob}
With the previous notations and assumptions,
let $J_\cV$ and $J_\cW$ denote the range functions of the SI subspaces $\cV$ and $\cW$, respectively. 
Then,
\begin{enumerate}
\item $\cW^\perp$ is a SI subspace with range function $J_{\cW^\perp}(x)=[J_{\cW}(x)]^\perp$ for a.e. $x\in \T^k$;
\item If $Q=P_{\cV//\cW^\perp}$ then $Q$ is a shift preserving operator;
\item $J_\cV(x)\oplus J_{\cW}(x)^\perp=\ell^2(\Z^k)$  and 
$[ Q]_x=P_{J_\cV(x)//J_{\cW}(x)^\perp}$ for a.e. $x\in\T^k$.
\item $E(\cG)\in \duv  \iff \Gamma\cG(x)\,\ \text{is}\,\ J_{\cV}(x)-\text{dual of}\,\ \Gamma\cF(x)$, for a.e $x\in \T^k $.
\QED
\end{enumerate}

\end{lem}

\begin{rem}\label{rem:Dixmier angle} Let $\cS$ and $\cT$ be closed subspaces of $L^2(\R^k)$. In order to characterize when the (algebraic) sum of these subspaces is a closed subspace we recall the Dixmier angle between $\cS$ and $\cT$, denoted by $[\cS,\,\cT]_D\in[0,\pi]$, given by 
\beq\label{def Dixmier angle}
\cos[\cS, \cT]_D=\sup\{|\langle v, \, w\rangle|, v\in \cS_1, w\in \cT_1\}\,,
\eeq where $\cS_1= \{f\in \cS: \|f\|=1\}$ (and similar for $\cT_1$). It is well known (see \cite{Deu})
that $[\cS, \cT]_D>0$ if and only if $\cS\cap\cT=\{0\}$ and $\cS+\cT$ is a closed subspace of $L^2(\R^k)$.

\pausa
Assume further that $\cS\oplus \cT=L^2(\R^k)$ and let $Q=P_{\cS//\cT}$ be the corresponding oblique projection. Then 
(see \cite{Deu}) 
\beq \label{norm Q}
\|Q\|=\frac{1}{\sin[\cS, \cT]_D}\  . \EOEP
\eeq
\end{rem}

\begin{pro}\label{pro: direct sum si}
Let $\cS, \cT \subseteq L^2(\R^k)$ be SI subspaces of $L^2(\R^k)$.
The following statements are equivalent:
\ben
\item  $\cS \oplus \cT^{\perp} = L^2(\R^k)$;
\item $J_{\cS}(x) \oplus J_{\cT}(x) ^{\perp}= \ell^2(\Z^k)$ for a.e. $x\in\T^k$ and $\esssup_{x\in\T^k}\|P_{J_\cS(x) // J_\cT(x)^{\perp}}\|<\infty$;
\item $J_\cS(x)^\perp\cap J_\cT(x)=\{0\}$ and $\essinf_{x\in\T^k}[J_{\cS}(x), J_\cT(x)^{\perp}]_D>0$.
\een
In this case we have that $ [{\cS},\cT^{\perp}]_D= \essinf_{x\in\T^k} [J_{\cS}(x), J_\cT(x)^{\perp}]_D\,.$
\qed\end{pro}
\pausa
As an application of the previous results we compute the exact value of the aliasing norm (see \cite{YEldar1,Janssen}) in terms of the relative geometry of the FSI subspaces $\cV$ and $\cW$. Indeed, recall that the aliasing norm corresponding to the consistent sampling $$ f\mapsto \tilde f= P_{\cW//\cV^\perp} f \ , \quad \text{ for } f\in L^2(\R^k)$$ denoted $A(\cV,\,\cW)$, is given by 
\beq \label{def: alias}
A(\cV,\,\cW)=\sup_{e\in \cW^\perp}\frac{\|P_{\cW//\cV^\perp} e\|}{\|e\|}=\|P_{\cW//\cV^\perp} \,P_{\cW^\perp}\|\,.
 \eeq
 The aliasing norm is a measure of the incidence of $\cW^\perp$ in the consistent sampling induced by $P_{\cW//\cV^\perp}$ and it plays a role in applications of oblique duality. 
 \begin{fed} \rm\label{defi apertur}
Let $\cS$, $\cT\subset L^2(\R^k)$ be closed subspaces. We define the aperture between $\cS$ and $\cT$, denoted $[\cS,\cT]^a\in [0,\pi/2]$, as the angle given by
\beq
\cos([\cS,\cT]^a)=\inf_{f\in\cT,\, \|f\|=1 }\|P_\cS f\|\, .
\EOEP
\eeq 
\end{fed}
\begin{rem}\label{rela entre apertur y angulo}
With the notations of Definition \ref{defi apertur}, we point out that the aperture $[\cS,\cT]^a$ coincides with the notion of angle between the subspaces $\cS$ and $\cT$ as defined in \cite{UnAl} (and $\cos([\cS,\cT]^a)$ is also known as the infimum cosine angle from \cite{KLL05}). 
 It is known that the following relation holds (see \cite{KLL05,KLL06}): 
$$\cos([\cS,\cT]_D)^2=1-\cos([\cS,\cT]^a)^2 \ \implies \ [\cS,\cT]^a=\pi/2 - [\cS,\cT^\perp]_D\,.$$
Hence, using the relations above and Proposition \ref{pro: direct sum si} (see also \cite{KLL05}) we get that if 
$\cS$, $\cT\subset L^2(\R^k)$ are SI subspaces such that $\cS\oplus\cT^\perp=L^2(\R^k)$ then 
\beq\label{rel angle SI}
[\cS,\cT]^a=\esssup_{x\in\T^k} [J_\cS(x),J_\cT(x)]^a<\pi/2\, .
\EOEP\eeq
\end{rem}

\pausa
Consider again the notations of Definition \ref{defi apertur} and assume further that $L^2(\R^k)=\cS\oplus\cT^\perp$. Then, using Remarks \ref{rem:Dixmier angle} and \ref{rela entre apertur y angulo} we see that $$\|P_{\cS//\cT^\perp}\|=\frac{1}{\sin[\cS, \cT^\perp]_D}= \frac{1}{\cos[\cS, \cT]^a}\,.$$
From this we obtain the following upper bound for the aliasing norm (see \cite{UnZe}) $$ A(\cS,\,\cT)\leq \|P_{\cS//\cT^\perp}\|=\frac{1}{\cos[\cS, \cT]^a}\,.$$ Notice that this known bound is not sharp; indeed, if we take $\cS=\cT$ then 
$A(\cS,\,\cT)=0$ but $\cos[\cS, \cT]^a=1$.

\pausa
Next we compute the exact value of the aliasing norm.
\begin{pro}\label{pro comp aliasing} With the previous notations and assumptions,
the aliasing norm $A(\cV,\cW)$ corresponding to the FSI oblique pair $(\cV,\cW)$ is given by 
 $$A(\cV,\,\cW)=\tan([\cV,\,\cW]^a)\, . $$ 
\end{pro}
\proof
 Notice that by assumption  $J_\cV(x)$ and $J_\cW(x)$ are finite dimensional subspaces of $\ell^2(\Z^k)$ and, by Proposition \ref{pro: direct sum si}, we see that $J_\cV(x)^\perp\oplus J_\cW(x)=\ell^2(\Z^k)$, for a.e. $x\in\T^k$. Hence, we can apply the results from \cite{BMS14}, and conclude that 
 $$
 A(J_\cV(x),\, J_\cW(x))=\| P_{J_\cW(x)//J_\cV(x)^\perp} \ P_{J_\cW(x)^\perp}\|= \tan([J_\cV(x),J_\cW(x)]^a) \ , \quad \text{ for a.e. } x\in\T^k\,.
 $$
Therefore, using Remark \ref{rela entre apertur y angulo}, we get that 
\beq
A(\cV,\,\cW)= \|P_{\cW//\cV^\perp} \,P_{\cW^\perp}\|=\esssup_{x\in\T^k}\tan([J_\cV(x),J_\cW(x)]^a)= \tan([\cV,\,\cW]^a)\, . \QEDP \eeq
 
\begin{conj}
 We conjecture that Proposition \ref{pro comp aliasing} holds for the consistent sampling corresponding to an oblique decomposition $\cS\oplus \cT^\perp=\hil$ in an arbitrary Hilbert space $\hil$. By the results from \cite{BMS14} the conjecture holds for finite dimensional $\cS$ and $\cT$.  By Proposition \ref{pro comp aliasing} this conjecture holds for some infinite dimensional subspaces $\cS$ and $\cT$ as well.\EOE
\end{conj}

\subsection{Fine spectral structure of SG oblique duals}\label{sec 4.2}

\pausa
Let $E(\cG)\in \duv$ and let $S_{E(\cG)}$ denote the frame operator of $E(\cG)$. Recall that in this case $S_{E(\cG)}$ is a shift preserving (SP) operator such that $ [S_{E(\cG)}]_x=S_{\Gamma \cG(x)}$ for a.e. $x\in \T^k$  
and the fine spectral structure of $E(\cG)$ is the function 
%$\T^k\ni x\mapsto \lambda(S_{\Gamma \cG(x)}) = (\lambda_j(S_{\Gamma \cG(x)}))_{j\in\N} \,$, 
$\T^k\ni x\mapsto (\lambda_j([S_{E(\cG)}]_x)\,)_{j\in\N}\in \big(\,\ell^1_+(\Z^k)\,\big)\da$, that consists of 
the sequence of eigenvalues of the positive finite rank operator $[S_{E(\cG)}]_x=S_{\Gamma \cG(x)}$, counting multiplicities and arranged in non-increasing order for a.e. $x\in\T^k$ (see Remark \ref{fine la}). 

\pausa
In the next result we consider the measurable function $d:\T^k\rightarrow \{0,\ldots,n\}$ given by $d(x)=\dim J_\cW(x)= \dim J_\cV(x)$ for a.e. $x\in\T^k$.

\begin{lem}\label{pro factoriz}
Let $\cG=\{g_i\}_{i\in\In}\in\cV^n$ be such that $E(\cG)$ is a frame for $\cV$. Let 
$B\in L(L^2(\R^k))^+$ be a shift preserving operator such that $R(B)\subseteq \cV$. Then, there exists $\cZ=\{z_i\}_{i\in\In}\in\cV^n$ such that $B=S_{E(\cZ)}$ and $T_{E(\cG)}\,T_{E(\cZ)}^*=0$ if and only if $\rk([B]_x )\leq n-d(x)$ for a.e. $x\in \T^k$. 
\end{lem}
\proof
First notice that by considering a convenient finite partition of $\T^k$ into measurable sets we can assume, without loss of generality, that $d(x)=d\in \N$ for a.e. $x\in \T^k\,$. Notice that in this case $n\geq d$.
 By Lemma \ref{lem spect represent ese} there exist measurable vector fields $v_j:\T^k\rightarrow \ell^2(\Z^k)$
for $j\in\In$ such that, if 
%$\lambda(S_{\Gamma \cG(\cdot)})=(\la_j(\cdot))_{j\in\N}$ 
$\T^k\ni x\mapsto (\lambda_j(x)\,)_{j\in\N}\in \big(\,\ell^1_+(\Z^k)\,\big)\da$ denotes the fine spectral structure of $E(\cG)$, then 
\beq\label{repre espec S}
[S_{E(\cG)}]_x 
=\sum_{j\in\IN{d}}\lambda_j(x)\ v_j(x)\otimes v_j(x) 
\ , \quad \text{ for a.e. } x\in\T^k\,.
\eeq
Moreover, in this case $\{v_j(x)\}_{j\in\IN{d}}$ is an ONB of $J_\cV(x)$ for a.e. $x\in\T^k$.
Assume that  $B\in L(L^2(\R^k))^+$ is a shift preserving operator such that $R(B)\subseteq \cV$ and such that $\rk([B]_x )\leq n-d$ for a.e. $x\in \T^k$. Since $[B]_x\in L(\ell^2(\Z^k))^+$ is such that $R([B]_x)\subseteq J_{\cV}(x)$ for a.e. $x\in\T^k$ 
then, using the measurable vector fields $\{v_j\}_{j\in\IN{d}}$ as above (indeed, the measurable field of matrix representations 
of $[B]_x$ with respect to $\{v_j(x)\}_{j\in\IN{d}}$ and the results from \cite{RS95}) we get measurable fields $w_j:\T^k\rightarrow \ell^2(\Z^k)$ for $j\in\IN{d}$, such that 
$\{w_j(x)\}_{j\in\IN{d}}$ is an ONB of $J_\cV(x)$ and $[B]_x\,w_j(x)
=\la_j([B]_x)\, w_j(x)$ for $j\in\IN{d}\,$ and a.e. $x\in\T^k$. 
In particular, we see that 
$$ 
[B]_x^{1/2}=\sum_{j=1}^{\min\{d,n-d\}}\la_j([B]_x)^{1/2}\, w_j(x)\otimes w_j(x)\ , \quad \text{for a.e. } x\in\T^k \ .
$$
Consider the measurable field of positive semidefinite matrices $G_\cG:\T^k\rightarrow \matn$ given by $G_\cG(x)=(\langle \Gamma g_i(x),\,\Gamma g_j(x)\rangle)_{i,\,j\in\I_n}$, $x\in\T^k$. Again by \cite{RS95}, there exist measurable field of vectors $u_j:\T^k\rightarrow \C^n$ for $j\in\I_n$ 
such that for a.e. $x\in\T^k$ we have that $\{u_j(x)\}_{j\in\I_n}$ is an ONB of $\C^n$, $G_\cG(x) \, u_j(x)=\la_j(x)\, u_j(x)$ for $j\in\I_d$ and 
$G_\cG(x) \, u_j(x)=0$ for $d+1\leq j\leq n$ (since $G_\cG(x)$ and $[S_{E(\cG)}]_x$ have the same strictly positive eigenvalues).

\pausa
Let $V:\T^k\rightarrow L(\C^n,\ell^2(\Z^k))$ be the measurable field of partial isometries given by 
$$
V(x) \, u_j(x) =\left\{ 
\begin{array}{ccc}
w_{j-d} & {\rm if} & d+1\leq j\leq d+\min\{d,n-d\}\,,    \\
0 & {\rm } & \text{otherwise}.   
\end{array}
\right.
$$
Hence $V(x)\,V^*(x) $ is the orthogonal projection onto $\text{span}\{w_j(x):\ j\in \IN{\min\{d,n-d\}}\}$ and thus $[B]_x\,V(x)\,V^*(x)=[B]_x$ for a.e. $x\in\T^k$; on the other hand, 
$R(V^*(x)\,) =\ker V(x) \orto \inc  \text{span}\{u_j(x):d+1\le j\le n\} \implies T_{\Gamma\cG(x)}V^*(x)=0$ for a.e. $x\in \T^k$.

\pausa
For $i\in\In$ we set $z_i\in \cV$ determined uniquely by $\Gamma z_i(x)
=[ B]_x^{1/2}\,V(x) e_i$ for a.e. $x\in\T^k$, where $\{e_i\}_{i\in\In}$ denotes the canonical ONB of $\C^n$. If we set $\cZ=\{z_i\}_{i\in\In}$ then $T_{\Gamma \cZ(x)} =[ B^{1/2}]_x\,V(x)$ for a.e. $x\in \T^k$; 
hence, using Eq. \eqref{eq:fourier}, we see that 
$$ 
[T_{E(\cG)}\,T_{E(\cZ)}^*]_x=T_{\Gamma \cG(x)} T_{\Gamma \cZ(x)}^*
= T_{\Gamma \cG(x)}V^*(x) \,[ B^{1/2}]_x=0 \ , \quad \text{ for a.e. } x\in\T^k\,.$$
On the other hand, notice that 
$$
[ S_{E(\cZ)}]_x=S_{\Gamma\cZ(x)}
=[B]^{1/2}_x \, V(x)\, V^*(x)\,[B]^{1/2}_x
=[B]_x\ , \quad \text{ for a.e. } x\in\T^k\ .
$$
Conversely, assume that
there exists $\cZ=\{z_i\}_{i\in\In}\in\cV^n$ such that $B=S_{E(\cZ)}$ and $T_{E(\cG)}\,T_{E(\cZ)}^*=0$. Then, by Eq. \eqref{eq:fourier}, we get that 
$0=T_{\Gamma \cG(x)}\ T_{\Gamma \cZ(x)}^*$ and hence $\rk(T_{\Gamma \cZ(x)}^*)\leq n-\rk(T_{\Gamma \cG(x)})=n-d(x)$
for a.e. $x\in\T^k$. Therefore,
\beq
\rk\, [S_{E(\cZ)}]_x=\rk(S_{\Gamma\cZ(x)})=\rk(T_{\Gamma \cZ(x)})=\rk(T_{\Gamma \cZ(x)}^*)\leq n-d(x) \quad \text{for a.e. } x\in \T^k\ .
\QEDP\eeq

\begin{fed}\label{el conjunto U}\rm
Let $\cG=\{g_i\}_{i\in\In}$ be such that $E(\cG)$ is a frame for $\cV$ with frame operator $\Sast=S_{E(\cG)}$. Recall that $d(x)=\dim J_\cV(x)$ for $x\in\T^k$. Then, we consider 
\beq
U_\cV(E(\cG)\,) = 
\Big\{\Sast + B:B\in L(L^2(\R^k))^+ \text{ is SP},\ R(B)\subset \cV,\
\rk([B]_x)\leq n-d(x) \, , \, \text{for a.e. } x\in \T^k  \, \Big\} . 
\EOEP\eeq
\end{fed}

\begin{pro}\label{pro SD}
Let $E(\cF)^\#_\cV=\{T_\ell \, f^\#_{\cV,\,i}\}_{(\ell,i)\in\Z^k\times \In}$ 
denote the canonical $\cV$-dual of $\cF$. Then,
\beq \label{ident s dual}
\{S_{E(\cG)}:\ E(\cG)\in \duv\} = U_\cV({E(\cF)^\#_\cV}) \ .
\eeq
\end{pro}
\proof
Let $\cG=\{g_i\}_{i\in\In}\in\cV^n$ be such that $E(\cG) \in \duv$. Let $\cZ=\{z_i\}_{i\in\In}\in\cV^n$ be given by 
 $z_i=g_i-f^\#_{\cV,\,i}$ for $i\in\In\, $. 
 Then $E(\cZ)=\{T_\ell \, z_i\}_{(\ell,\,i)\in\Z^k\times \In}$ 
 is a Bessel sequence in $\cV$ such that $T_{E(\cG)}=T_{E(\cF)^\#_{\cV}}+T_{E(\cZ)}$. In this case we have that $T_{E(\cZ)}\,T_{E(\cF)}^*=0$ and therefore $T_{E(\cZ)}\,T_{E(\cF)^\#_{\cV}}^*=0$,
 since $R(T_{E(\cF)}^*)=R(T_{E(\cF)^\#_{\cV}}^*)$. Thus, 
$$ 
S_{E(\cG)}=(T_{E(\cF)^\#_{\cV}}+T_{E(\cZ)})\,(T_{E(\cF)^\#_{\cV}}+T_{E(\cZ)})^* =S_{E(\cF)^\#_{\cV}}+ S_{E(\cZ)}\ . 
$$   
We conclude that $B= S_{E(\cZ)}\in L(L^2(\R^k))^+$ is SP, $R(S_{E(\cZ)})\subset \cV$ and, by Lemma \ref{pro factoriz}, that $\rk\, [S_{E(\cZ)}]_x\leq n-d(x)$ for a.e. $x\in \T^k$. 
 
\pausa
Conversely, if $S\in U_\cV({E(\cF)^\#_\cV})$ then $S=S_{E(\cF)^\#_\cV}+B$, where
$B\in L(L^2(\R^k))^+$ is SP,  $R(B)\subset \cV$ and $\rk([B]_x)\leq n-d(x)$ for a.e. $x\in X$. By Lemma \ref{pro factoriz} we see that there exists $\cZ=\{z_i\}_{i\in\In}$ such that $T_{E(\cZ)}\,T_{E(\cF)}^*=0$ and $B=S_{E(\cZ)}$.
If we let $\cG=\{g_i\}_{i\in\In}$ be given by $g_i=f^\#_{\cV,\,i}+z_i$ for $i\in\In$, then $E(\cG)$ is a Bessel sequence in $\cV$ such that $T_{E(\cG)}=T_{E(\cF)^\#_\cV}+T_{E(\cZ)}$. 
Using that $T_{E(\cZ)}\,T_{E(\cF)^\#_\cV}^*=0$ we conclude, as before, that 
\beq
S_{E(\cG)}=S_{E(\cF)^\#_{\cV}}+ S_{E(\cZ)}=S_{E(\cF)^\#_{\cV}}+ B=S   
\QEDP\eeq

\pausa
Proposition \ref{pro SD} shows that the set of frame operators of SG $\cV$-duals of a fixed frame $\cF$ can be described in terms of the additive model $U_\cV(E(\cF)^\#_\cV)$ 
introduced in Definition \ref{el conjunto U}. It turns out that the fine spectral structure of the elements of $U_\cV(E(\cF)^\#_\cV)$ can be described using a natural extension of the Fan-Pall interlacing theorem for measurable fields of positive matrices. We develop both results in the Appendix section (see Theorems \ref{teo Fan-Pall medible} and \ref{estructdelU}). As a consequence we obtain the following 

\begin{teo}[Fine spectral structure of $\cV$-duals]\label{teo struc espec}
Let $E(\cF)^\#_\cV$
be the canonical $\cV$-dual frame of $E(\cF)$. Denote the fine spectral structure of $E(\cF)^\#_\cV$ by $\T^k\ni x\mapsto (\lambda_{\cV,\,i}^\#(x))_{i\in\N}$,
%Denote by $\Sast =  {S}_{E(\cF)^\#_{\cV}}$ and 
%by $\lambda_{\cV,\,i}^\#(x)=\lambda_i(\Bx)$, $i\in \N$,
 $x\in\T^k$.
Let $m$ be the measurable function given by
$m(x)=2d(x)-n$, for $x\in\T^k$.
 Given a measurable function $\mu:\mathbb T^k\rightarrow (\ell^1(\N)^+)\da$ (decreasing sequences)  
 described as $\mu=(\mu_i)_{i\in\N}\,$,  the following are equivalent:
\ben
\item There exists $E(\cG)\in \duv$ such that  $\mu (x) = 
\la (\,[S_{E(\cG)}]_x)=\la(S_{\Gamma \cG(x)})  $ for every  $ x\in\T^k$.
\item For a.e. $x\notin\text{Spec}(\cV)$, $\mu(x) =0$. For a.e. $x\in\text{Spec}(\cV)$, $\mu_i(x)=0$ for $i\geq  d(x)+1$ and 
\begin{enumerate}
\item in case that $m(x)\leq 0$, then  $\mu_i(x)\geqp \lambda_{\cV,\,i}^\#(x)$ for  $i\in \I_{d(x)}$;
\item in case that $m(x)\ge 1$, then  $\mu_i(x)\geqp \lambda_{\cV,\,i}^\#(x)$ for $i\in \I_{d(x)}$
and  
$$
\mu_{n- d(x)+i}(x) = \mu_{d(x)-m(x)+i}(x)\le \lambda_{\cV,\,i}^\#(x) \peso{for} i\in \I_{m(x)}  \ .
$$
\end{enumerate}
\een
\end{teo}
\proof It follows from Proposition \ref{pro SD} and Theorem \ref{estructdelU}. \QED

\pausa
As a consequence of the description of the fine spectral structure of $\cV$-duals of $E(\cF)$ we characterize the existence of {\it tight} $\cV$-duals of $E(\cF)$ that are {\it shift generated} (compare with \cite{DHan08}).

\begin{cor}\label{cor: existence of tight duals}
With the notations of Theorem \ref{teo struc espec} then there exists a $c$-tight $\cV$-dual $E(\cG)\in \duv$ if and only if
\ben
\item $S_{E(\cF)^\#_\cV}\leq c\cdot P_\cV$;
\item $\rk (\,[c\cdot P_\cV-S_{E(\cF)^\#_\cV}]_x)\leq \min\{d(x),\, n-d(x)\}$ for a.e. $x\in\text{ Spec}(\cV)$.
\een
\end{cor}
\begin{proof}  
Theorem \ref{teo struc espec} imply that there exists a $\cV$-dual $E(\cG)\in \duv$ such that $S_{E(\cG)}=c\cdot P_\cV$ if and only if 
$c\geq \lambda_{\cV,\,i}^\#(x)$
%\lambda_i( \Bx)$ 
for $i\in\IN{d(x)}$ and 
%$\lambda_i( \Bx)=c$
$\lambda_{\cV,\,i}^\#(x)=c$
 for $i\in\IN{m(x)}$ whenever $m(x)=2d(x)-n\geq 1$, for a.e. $x\in\text{Spec}(\cV)$.
These last two conditions are equivalent to the fact that $c\cdot P_\cV\geq S_{E(\cF)^\#_\cV}$ and 
$$
\rk(\,[c\cdot P_\cV-S_{E(\cF)^\#_\cV}]_x)\leq d(x)-m(x)=n-d(x) 
\peso{whenever} m(x)\geq 1 \ . 
$$
Also notice that in case $m(x)\leq 0$ then $n-d(x)\geq d(x)=\dim J_\cV(x)$. The proof follows from these remarks.
\end{proof}

\begin{rem}\label{dicotomia dual}
Consider the notations of Theorem \ref{teo struc espec}. As a consequence of Corollary \ref{cor: existence of tight duals}, we get the following dichotomy related with the existence of tight oblique $\cV$-duals of $E(\cF)$:
\ben
\item If $n\geq 2\,d(x)$ for a.e. $x\in\T^k$ then for every $c\geq \|S_{E(\cF)^\#_\cV}\|$ there exists $E(\cG)\in \duv$ that is a $c$-tight frame for $\cV$.
\item If there exists $N\inc \T^k$ with positive Lebesgue measure such that $n< 2\,d(x)$ for a.e. $x\in N$ and there exists a $c$-tight frame $E(\cG)\in\duv$ then $c=\|S_{E(\cF)^\#_\cV}\|$. \EOE
\een
\end{rem}

\section{Applications: optimal oblique SG-duals with norm restrictions}\label{sec applic}

As before, we consider two FSI subspaces $\cV$ and $\cW$ such that $\cW\orto \oplus \cV=L^2(\R^k)$ and $\cF=\{f_i\}_{i\in\In}\in\cW^n$ such that $E(\cF)$
is a frame for $\cW$.

\pausa
As a consequence of the description  of the fine spectral structure of elements in $\duv$, we see that 
the canonical $\cV$-dual is optimal with respect to several criteria. Nevertheless, in applied situations, the canonical dual might not be the best choice: for example, we can be interested in duals of $E(\cF)$ such that the spectrum of their frame operators are as concentrated as possible. Ideally, we would search for tight dual frames for $E(\cF)$, although Corollary \ref{cor: existence of tight duals} shows that there are restrictions for the existence of such duals.

\pausa
In order to search for alternate $\cV$-duals that are spectrally more stable, we proceed as follows: for 
$w \geq \sum_{i\in\In}\|f^\#_{\cV,\,i}\|^2$, where $E(\cF)^\#_\cV=\{T_\ell\,f^\#_{\cV,\,i}\}_{(\ell,\,i)\in\Z^k\times \In}$, we consider
$$\duvw = 
\cD_{{\cV}, \, w}^{SG}(E(\cF))\igdef
\big\{ E(\cG) \in \duv: \ 
 \cG = \{g_i\}_{i\in\In} \ \ \text{and} \ \  
 \sum_{i\in\In}\|g_i\|^2\geq w 
\big\}\ .$$  
Notice that if $w>\sum_{i\in\In}\|f^\#_{\cV,\,i}\|^2$ 
then $E(\cF)^\#_\cV \notin \duvw$ and therefore, it is natural to ask whether there is an optimal dual fulfilling the previous requirements. Using the identity 
\beq\label{eq ident tra norm}
\sum_{i\in\In}\|g_i\|^2= \int_{\T^k}\sum_{i\in\In} \|\Gamma g_i(x)\|^2\ dx= \int_{\T^k} \tr(\,[S_{E(\cG)}]_x)\ dx= \int _{\T^k} \ \sum_{i\in\N}\mu_i(x) \ dx 
 \eeq
where $\lambda(\,[S_{E(\cG)}]_x)=(\mu_i(x))_{i\in\N}$ for a.e. $x\in\T^k$, we see that Theorem \ref{teo struc espec}  gives a complete solution to a frame design problem in the sense that it allows to get a complete description of the eigenvalue lists of the frame operators of elements in $\duvw$. It is then natural to seek for those oblique SG-duals $E(\cG)\in\duvw$ that minimize the convex potentials $P^\cV_\varphi$, for $\varphi\in\convf$; in order to deal with this problem we first examine a construction known as water-filling in terms of submajorization, in the general context of measure spaces (see Theorem \ref{teo op water}). We then apply these results together with the properties of submajorization and results from matrix analysis to conclude that there are structural optimal duals with norm restrictions.  These optimal solutions are obtained in terms of a non-commutative water-filling construction. 

\subsection{Water-filling in measure spaces}

The water-filling construction goes back to the work of Shanon \cite{Shann}, as the solution of an optimal spectral allocation problem (see \cite{CTbook}). The water-filling strategy has also been the main tool in the design of channels with optimal capacity (see \cite{Telat} and the more recent work on iterative water-filling techniques \cite{PF05,CPB09}).

\pausa
As a first step towards an extension of this construction, we examine its scalar counter-part in the general context of measure spaces. In the next section we show that the water-filling technique produces optimal solutions in the general (non-commutative) context of measurable fields of positive semidefinite matrix valued functions.

\pausa
Throughout this section the triple $(X,\cX,\mu)$ denotes 
a probability space. Recall that we  denote by $L^\infty(X,\mu)^+ 
= \{f\in L^\infty(X,\mu): f\ge 0\}$.

\begin{fed}[Water-filling at level $c$]\label{def: wf}\rm
Let $f\in L^\infty(X,\mu)^+ $. Given $c \geq \essinf f\geq 0$ we consider $f_c\in L^\infty(X,\mu)^+ $ given by $f_c= \max\{f,\,c\}=f + (c-f)^+$, where $g^+$ denotes the positive part of a real function $g$. 
\EOE
\end{fed}

\pausa
In order to study the submajorization properties of the function $f_c$ obtained by the water-filling construction as above, we consider the following result in which we obtain a simple relation between the decreasing rearrangements of $f$ and $f_c$.

\begin{lem}\label{pro: prop. reo.} Let $f\in L^\infty(X,\mu)^+ $ and let $c \geq \essinf f\geq 0$. 
Consider the number  
\beq\label{eq lem reord}
s_0= \mu\{x\in X:\ f(x)>c\} \ . \peso{Then}
f_c^*(s)=\left\{
  \begin{array}{ccc}
     f^*(s) & if & 0\leq s <s_0 \,; \\
      c & if & s_0\leq s \leq 1\, .
   \end{array}
	\right.
	\eeq
\end{lem}
\begin{proof}
Notice that by Eq. \eqref{eq:reord}, for $0\leq s<s_0$ we have that 
\begin{eqnarray*}
f^*(s)&=& \sup \, \{ t\in \R^+ : \ \mu \{x\in X:\ f(x)>t\} >s\}\\
&=&   \sup\, \{ t\in \R^+ : \ \mu \{x\in X:\ f(x)>t\} >s  
\ \text{ and } \ t\geq c \}\\
&=& \sup\, \{ t\in \R^+ : \ \mu \{x\in X:\ f_c(x)>t\} >s\ \text{ and } \ t\geq c \}\\
&=& \sup\, \{ t\in \R^+ : \ \mu \{x\in X:\ f_c(x)>t\} >s\}=f^*_c(s) \ .
\end{eqnarray*}
It is straightforward to see that if $s_0\leq s\leq 1$ then $f^*_c(s)=c$.
\end{proof}

\pausa
In order to prove Theorem \ref{teo op water} below, we shall need 
an explicit statement of some re-parametrized versions of the basics results of section \ref{2.3}:

\begin{lem}\label{lem prop reord ct}
Let $a\coma b \in \R$ be such that $a<b$ and let  $k\in L^\infty([\,a\coma b\,],\, \nu)^+$ 
be a non-increasing right continuous function, where $\nu=(b-a)^{-1}\, dt$ is the normalized Lebesgue measure on $[a,b]$. 
 Then 
\ben
\item 
The decreasing rearrangement 
$k^*(t) = k\big(\, (b-a)\, t + a\,\big) $ for every $t \in [0\coma 1)$. 
\item Fix a constant $c\in \R$. Then 
$$ 
(b-a)\, c \le \int _a^b \, k(t) \, dt \implies (s-a)\, c \le \int _a^s \, k(t) \, dt \peso{for every}
s \in [a\coma b] \ .
$$
\een
\end{lem}
\proof Straightforward. \QED

\pausa
With the notations of Lemma \ref{lem prop reord ct} above, notice that item 2. is a restatement (using the re-parametrization from item 1) of the submajorization inequalities corresponding to 
$c\prec_w k$ in $([a,b],\,\nu)$ whenever $c\leq \int_{[a,b]} k\ d\nu$ (see Example \ref{exa: integral}).

\begin{rem}\label{defi ct}
 Let $f\in L^\infty(X,\mu)^+$ and consider $\phi: [\essinf f, \infty)\to \R^+$ given by
$$\phi(c)=\int_X f_c\ d\mu= \int_X f(x) + (c-f(x))^+\ d\mu(x)\,. $$ 
Then, it is easy to see that $\phi$ has the following properties:
\ben
\item $\phi(\essinf f)=\int_X f\ d\mu$ and $\lim _{c\to +\infty} \phi(c)= +\infty$;
\item $\phi$ is continuous and strictly increasing.
\een Hence, for every $w\geq \int_X f\ d\mu$ there exists a unique $c(w)=c\geq \essinf f$ such that 
\beq\label{eq. defi ct}
\phi(c(w))=w \peso{i.e.} \int_X f_{c(w)}\ d\mu=w\ . 
\eeq \EOE
\end{rem}

\begin{teo}[$\prec_w$-optimality of water-filling] \label{teo op water}
Let $f\in L^\infty(X,\mu)^+$, take $w\geq \int_X f\ d\mu$ and consider the constant  $c(w)=c$ as in Remark \ref{defi ct}.
 Then, for every
$h\in L^\infty(X,\mu)^+$, 
$$
f\leq h \py \int_X h\ d\mu\geq w  \implies f_c\prec_w h  \ .
$$
\end{teo}

\begin{proof} 
Assume that 
$f\leq h$ and $\int_X h\ d\mu\geq w$.
If we let $s_0=\mu\{x\in X:\ f(x)>c\}$ then, by Lemma \ref{pro: prop. reo.}, we have that Eq. \eqref{eq lem reord} holds. Thus, using Remark \ref{rem:prop rear elem}, we see that if $0\leq s < s_0$ then  
\beq \label{eq teo parte1}
\int_0^sf_c^*(t)\ dt=\int_0^sf^*(t)\ dt\leq \int_0^s h^*(t)\ dt\ .
\eeq
Fix now $s_0\leq s\leq 1$ and consider 
$$\al = \int_{0}^{s_0} h^*(t)\,\ dt - \int_{0}^{s_0} f_{c}^*(t)\,\ dt \geq 0 \py
k= h^*+ \frac{1}{1-s_0} \,\al \in L^\infty([0, 1],\ dt)^+\ .$$
Notice that $k$ is a non-increasing right continuous map. In this case we get that
\begin{eqnarray*}
\int_{s_0}^{1} k(t) \ dt&=&\int_{s_0}^{1} h^*(t)\,\ dt + \alpha \\ 
&=&\int_{s_0}^{1} h^*(t)\ dt + \left(\int_{0}^{s_0} h^*(t)\ dt - \int_{0}^{s_0} f_{c}^*(t)\,\ dt\right)\\
&\ge& w -(w - (1-s_0)\, c)=(1-s_0)\, c\ .
\end{eqnarray*}
Then, Lemma \ref{lem prop reord ct} (applied to the map $k|_{[s_0\coma 1]}\,$) implies that 
$$
(s-s_0)\, c \leq \int_{s_0}^{s} k(t)\ dt =
\int_{s_0}^{s} h^*(t)\ dt + \frac{s-s_0}{1-s_0} \ \alpha \fe s\in [s_0\coma 1] \ . 
$$ 
Hence, using the inequality above and Lemma \ref{pro: prop. reo.}, we conclude that for $s_0\leq s<1$
\begin{eqnarray*}
\int_{0}^{s} f_{c}^*(t) \ dt &=&\int_{0}^{s_0} h^*(t)\,\ dt -\alpha +(s-s_0)\, c \\ 
&\leq & \int_{0}^{s} h^*(t)\,\ dt + \left(\frac{s-s_0}{1-s_0} -1\right)\ \alpha\leq \int_{0}^{s} h^*(t)\,\ dt \ .
\end{eqnarray*}
This last fact together with Eq. \eqref{eq teo parte1} show that $f_c\prec_w h$.
\end{proof}

\pausa
Theorem \ref{teo op water} above implies a family of integral inequalities in terms of convex functions involving the water-filling of a function $f$ at level $c$. We will need these facts in order to show the optimality properties of the non-commutative version of waterfilling.

\begin{cor}\label{pregunta reord1.5}
With the notations of Theorem \ref{teo op water}, if $\varphi\in\convf$ is non-decreasing then
\beq\label{eq: ineq von varph}
 \int _X \varphi\circ h\ d\mu\geq \int_X \varphi\circ f_c\ d\mu\, .
 \eeq
If there is a non-decreasing $\varphi\in\convfs$ such that equality holds in Eq. \eqref{eq: ineq von varph}
then $h=f_c$.
\end{cor}
\begin{proof}
The first claim is a consequence of the submajorization relation in Theorem \ref{teo op water} and Theorem 
\ref{teo porque mayo}. If we further assume that $\varphi\in\convfs$ is such that equality holds in Eq. \eqref{eq: ineq von varph} then, by Proposition \ref{pro int y reo}, we see that $f_c^*=h^*$. Let $B=\{x\in X:\ f(x)>c\}$ so that $s_0=\mu(B)$. Then, it is straightforward to show that 
$$ 
(f\cdot 1_B)^*(s)=
\left\{
  \begin{array}{ccc}
     f^*(s) & \text{if} & s\in[0,s_0)\, ; \\
      0 & \text{if} & s\in[s_0,1)\, .
   \end{array}
	\right.
$$ Notice that, in particular, $(f\cdot 1_B)^*=1_{[0,s_0)}\cdot f_c^*$. 
On the other hand we have $(h\cdot 1_B)^*=1_{[0,\,s_0)}\cdot (h\cdot 1_B)^*$. Hence, 
since $h\geq h\cdot 1_B\geq f\cdot 1_B$, by Remark \ref{rem:prop rear elem}
 we have that $$h^*\geq (h\cdot 1_B)^*=1_{[0,s_0)}\cdot (h\cdot 1_B)^* \geq (f\cdot  1_B)^*=1_{[0,s_0)}\cdot f_c^*\quad \implies \quad  (h\cdot 1_B)^* =(f\cdot  1_B)^*\, .$$
Therefore, again by Remark \ref{rem:prop rear elem}, we get that $h\cdot 1_B =f\cdot  1_B=f_c\cdot  1_B> c\cdot  1_B$, where the last facts follow from Definition \ref{def: wf}. Finally, notice that 
$$
\mu(h^{-1}(\{c\}))=|(h^*)^{-1}(\{c\})| =|(f_c^*)^{-1}(\{c\})|=1-\mu(B) \ , 
$$
which shows that $h\cdot 1_{X\setminus B}=c\cdot 1_{X\setminus B}$ 
and hence $h=f_c\,$.
\end{proof}

\subsection{Optimal SG-duals with norm restrictions: NC water-filling}

In what follows we show the existence of {\it structural} optimal SG oblique duals of a fixed frame $E(\cF)$ with norm restrictions, as described at the beginning of Section \ref{sec applic}. That is, we explicitly construct a dual frame $E(\cG^{\rm op})\in \duvw$ such that for every $E(\cG)\in \duvw$ then 
$$ P_\varphi^\cV(E(\cG^{\rm op}))\leq P_\varphi^\cV(E(\cG)) \, ,$$
for every convex potential $P^\cV_\varphi$ associated to a non-decreasing $\varphi\in\convf$. Moreover, the arguments involved in this construction show that (structural) optimal SG oblique duals with norm restrictions  share several spectral properties. We end the section with a non-commutative counter-part of the water-filling construction for functions, that allows to describe the spectral and geometrical structure of optimal SG oblique duals in $\duvw$ in some detail.

\begin{teo}[Optimal duals in $\duvw$]\label{teo opt dual con norm}
Let $\cV$ and $\cW$ be FSI subspaces of $L^2(\R^k)$ such that $\cW\orto \oplus \cV=L^2(\R^k)$.
Let $\cF=\{f_i\}_{i\in \In}$ be such that $E(\cF)$ is a frame for $\cW$ and
$w> \sum_{i\in\In}\|f^\#_{\cV,\,i}\|^2$, where $E(\cF)^\#_\cV=\{T_\ell\,f^\#_{\cV,\,i}\}_{(\ell,\,i)\in\Z^k\times \In}$. Then, there exists $\cG^{\rm op}=\{g_i^{\rm op}\}_{i\in\In}\in\cV^n$ such that:
\begin{enumerate}
\item $E(\cG^{\rm op})\in \duvw$ and $\sum_{i\in\In}\|g_i^{\rm op}\|^2= w$.
\item For every $\cG=\{g_i\}_{i\in\In}$ such that $E(\cG)\in \duvw$ 
and every non-decreasing $\varphi\in \convf$ we have that
$$ P_\varphi^\cV(E(\cG^{\rm op}))\leq P_\varphi^\cV(E(\cG)) \ .$$ 
\end{enumerate}
\end{teo}
\begin{proof}
Let $d(x)=\dim J_\cV(x)=\dim J_\cW(x)$ for $x \in \T^k$. 
For each $i\in\In\,$, 
 let $X_i=d^{-1}(i)\subseteq \T^k$, 
$ p_i=|X_i|$ (the Lebesgue measure of $X_i$) and  
$r_i=\min\{n-i,\,i\}$. Since $E(\cF)$ is a frame for $\cW$ then $\text{Spec}(\cV)=\text{Spec}(\cW)=\cup_{i\in\I_n} X_i$.
Also, for $i\in \In$ and $j\in \I_{r_i}$ we consider the measure space $(X_{ij},\cX_{ij}, |\cdot|_{ij})$, where 
$X_{ij}=X_i$, $\cX_{ij}=\cX_i$ denotes the $\sigma$-algebra of Lebesgue measurable sets in $X_i$ and  
$|\cdot|_{ij}=|\cdot|_i$ denotes the Lebesgue measure in $X_i$.
Then, using Remark  \ref{rem const de esp}, we construct the measure space
$$(Y, \mathcal Y, \nu)= \bigoplus_{i\in \In} \bigoplus_{j\in \I_{r_i}}(X_{ij}, \mathcal Y_{ij}, |\cdot|_{ij}).
$$ 
In particular, $\nu(Y)=\sum_{i\in \In} r_i \cdot p_i$. We further consider the canonical inclusion maps $\eta_{i,j}:X_{i,j}\rightarrow Y$ for $i\in\In$ and $j\in\IN{r_i}$.

\pausa
Let $\cG=\{g_i\}_{i\in\In}$ be such that $E(\cG)\in \duv$. 
We shall denote by 
$A = S_{E(\cF)^\#_\cV}  $ and $S= S_{E(\cG)}\in L(L^2(\R^k))^+\,$. 
By Proposition \ref{pro SD},   
$S=S_{E(\cF)^\#_\cV}+B= \Sast+B$, 
for some $B\in L(L^2(\R^k))^+$ which is SP, $R(B)\subset \cV$ and  
$\rk([B]_x)\leq n-d(x)$ for a.e. $x\in\T^k$. Let $i\in\In$; using Lidskii's additive inequality (see \cite{Bhat}) we get that for a.e. $x\in X_i\,$
\beq \label{eq: Lidskii ineq fibers} 
\big(\,\lambda_{i-j+1}(\Bx )+ \lambda_{j}([B]_x) \,\big)_{j\in\IN{i}}\prec 
\big(\,\lambda_j(\,[S]_x) \,\big)_{j\in\IN{i}} \, ,
\eeq while $\lambda_j(\,[S]_x)=0$ for $j\geq i+1$. Notice that 
$R([B]_x)\subset J_\cV(x)$ and $\rk([B]_x)\leq n-i$.  
Therefore  $\rk([B]_x)\leq \min\{n-i,\, i\}=r_i\,$. Then, for $x\in X_i$ we have that
\beq \label{eq: dos cachos}
\lambda_{i-j+1}(\,\Bx )+ \lambda_{j}([B]_x)
=\left\{
  \begin{array}{ccc}
\lambda_{i-j+1}(\Bx ) + \lambda_{j}([B]_x )   
 & {\rm if} & 1\leq j\leq r_i\,;  \\
\lambda_{i-j+1}(\,\Bx ) & {\rm if} & r_i+1\leq j\leq i\,.  \\
   \end{array}
	\right.
\eeq
Now, Eq. \eqref{eq: Lidskii ineq fibers} together with Eq. \eqref{eq: dos cachos} 
imply that, for any $\varphi\in\convf$: for a.e. $x\in X_i$ then  
\beq\label{eq: preparando la cosa}
\sum_{j=1}^{r_i} \varphi(\lambda_{i-j+1}(\,\Bx ) + \lambda_{j}([B]_x))+
\sum_{j=r_i+1}^i \varphi(\lambda_{i-j+1}(\,\Bx )\,)
\leq \sum_{j\in\IN{i}} \varphi(\,\lambda_j( [S]_x )\, )\ .
\eeq
With the previous notations, we now consider the measurable function 
$h :Y\rightarrow \R^+$ defined as follows: 
for $x\in Y$, let $(i\coma j)\in  \I_n\times \I_{r_i}\,$ 
and $\tilde x\in X_{i,j}=X_i$ be (uniquely determined) such that
$\eta_{i,j}(\tilde x)=x$;
in this case we set $h (x)=\lambda_{i-j+1}(\,[ A]_{\tilde x}  ) + \lambda_{j}([B]_{\tilde x})$.
If we let $w_0=\sum_{i\in\In}\|f^\#_{\cV,\,i}\|^2$ and we assume  that $E(\cG)\in 
\duvw$ then, using Eq. \eqref{eq ident tra norm} we see that 
$$ 
\int_{\T^k} \tr([B]_x)\ dx=\int_{\T^k} \tr(\, [S]_x 
- 
\Bx ) \ dx \geq w-w_0\geq 0\ .
$$ 
Consider now the measurable function $f 
:Y\rightarrow \R^+\,$ given by  
$f(x) = \lambda_{i-j+1}(\,[ A]_{\tilde x} )$ for $\tilde x\in  X_{ij}=X_i\,$, with 
$(i\coma j)\in  \I_n\times \I_{r_i}\,$ such that $\eta_{i,j}(\tilde x)=x$. Arguing as in the proof of Theorem 
\ref{pro min del PF} we get that 
\beq \label{int f}
\int_Y f\ d\nu=\sum_{i\in\In}\int_{X_i} \sum_{j\in\IN{r_i}} \lambda_{i-j+1}
(\,\Bx)\ dx\,. 
\eeq
Moreover, the previous facts show that if $E(\cG)\in 
\duvw$ we have that $h 
\geq f$ and 
$$
\int_Y h 
(x)\ d\nu =\sum_{i\in\In}\,\sum_{j\in\IN{r_i}} \int_{X_i} 
(\lambda_{i-j+1}(\,\Bx) + \lambda_j([B]_x))\  dx
\geq (w-w_0)+ \int_Y f(x)\ d\nu\igdef w'\ .
$$
Let $c=c(\frac{w'}{\nu(Y)})$ be as in Remark \ref{defi ct} and consider $f_c$ as in Definition \ref{def: wf}, both with respect to the probability space $(Y,\cY,\tilde \nu)$, where $\tilde \nu=\nu(Y)^{-1}\,\nu$. By Corollary \ref{pregunta reord1.5} and the previous remarks we see that if $\varphi\in \convf$ is non-decreasing then 
\beq\label{eq: fc opti}
\int_Y\varphi\circ f_c\ d\nu\leq \int_Y\varphi\circ h 
\ d\nu\,.
\eeq
For $j\in\In$ we consider the measurable functions $\xi_j:\text{Spec}(\cV)\rightarrow \R^+$ defined 
as follows:
for $i\in\In$ and $x\in X_i\,$,  
\beq \label{eq const xi}
\xi_j(x)=\left\{
  \begin{array}{ccr}
     f_c(\eta_{ij}(x)) = \max \{c\coma \lambda_{i-j+1}(\, \Bx)\}& {\rm if} & 1\leq j\leq r_i  \\
     \lambda_{i-j+1}(\, \Bx) & {\rm if} & r_i+1 \leq j\leq i   \ \\
     0 & {\rm if} & i+1 \leq j\leq n\,   \\
   \end{array} \quad  , 
	\right. 
\eeq 
Notice that by construction, if $\varphi\in\convf$ then
\beq\label{eq: indent varphi fc}
\int_Y \varphi\circ f_c\ d\nu=\sum_{i\in\In} \int_{X_i} \ \sum_{j\in\IN{r_i}} \varphi(\xi_j(x))\ dx\,.  
\eeq
 Using the definition of $f$ and the properties of $f_c$ from Remark \ref{defi ct}, we see that if $x\in X_i\,$, then there exist $\lambda_1^{\rm op}(x)\geq \ldots\geq \lambda_{r_i}^{\rm op}(x)\geq 0$ such that 
\beq\label{eq: rela xi0}
 \xi_j(x)=\lambda_{i-j+1}(\,\Bx) + \lambda_j^{\rm op}(x) \peso{for every}
   j\leq r_i \  .
 \eeq
Let $\mu=(\mu_j)_{j\in\N}:\T^k\rightarrow \ell^1(\N)^+$ such that: $\mu_j(x)=0$ for $j\in\N$ whenever $x\in \T^k\setminus\text{Spec}(\cV)$, while for $i\in\In$ and $x\in X_i$  then $\mu_j(x)=0$ for $j\geq i+1$ and 
\beq\label{eq: rela xi}
(\mu_j(x))_{j\in\IN{i}}= [(\xi_j(x))_{j\in\IN{i}}]^\downarrow\,.
\eeq
Putting the previous remarks together we see that $\mu=(\mu_j)_{j\in\N}$ satisfies the conditions of item 2 in 
Theorem \ref{teo struc espec}. Thus, there exists 
$\cG^{\rm op}=\{g_i^{\rm op}\}_{i\in\In}$ such that
$E(\cG^{\rm op})\in \duv$ and 
$\lambda( \,[S_{E(\cG^{\rm op})}]_x)=(\mu_j(x))_{j\in\N}$ 
for a.e. $x\in\T^k$.
In this case, if we consider Eq. \eqref{eq ident tra norm}, use Eqs. \eqref{eq: rela xi0}, \eqref{eq: rela xi} and we take $\varphi(x)=x$ in Eq. \eqref{eq: indent varphi fc} we have that 
\begin{eqnarray*}
\sum_{i\in\In}\|g_i^{\rm op}\|^2&=&\int_{\T^k} \ \sum_{j\in\N}\mu_j(x)\ dx=\sum_{i\in\In} \int_{X_i} \left( \sum_{j\in\IN{r_i}} \xi_j(x) + \sum_{j=r_i+1}^i 
\lambda_{i-j+1}(\Bx)\right)\ dx \\ &=& 
\int_Y f_c\ d\nu+ \sum_{i\in\In} \int_{X_i} \sum_{j=r_i+1}^i 
\lambda_{i-j+1}(\,\Bx)\ dx \\ &=& 
 (w-w_0) 
+ \int_{\T^k} \tr(\,\Bx)\ dx=w \ ,
\end{eqnarray*}
where we have also used the relation in Eq. \eqref{int f} above. In particular, $\cG^{\rm op}$ satisfies item 1. in the statement. Now, if $E(\cG)\in \duvw$, using Eqs. \eqref{eq: preparando la cosa}, \eqref{eq: fc opti}, \eqref{eq const xi} and \eqref{eq: rela xi} then,
\begin{eqnarray*}
P_\varphi^\cV(E(\cG))&\geq& \int_Y \varphi\circ h \ d\nu
+\sum_{i\in\In} \int_{X_i} \sum_{j=r_i+1}^i \varphi\circ \lambda_{i-j+1}(\Bx)\ dx \\ 
&\geq & \int_Y \varphi\circ f_c\ d\nu+\sum_{i\in\In} \int_{X_i} \sum_{j=r_i+1}^i \varphi\circ \lambda_{i-j+1}(\Bx)\ dx =P_\varphi^\cV(E(\cG^{\rm op}))
\end{eqnarray*} where we have also used Eq. \eqref{eq: indent varphi fc} and the fact that $\la(\,[(S_{E(\cG^{\rm op})}]_x)=\mu(x)$ for a.e. $x\in\T^k$.
\end{proof}

\begin{cor}[Essential uniqueness of optimal $\cV$-duals with norm restrictions]\label{coro unico}
With the notations of Theorem \ref{teo opt dual con norm} and its proof, assume that $\cG=\{g_i\}_{i\in\In}$ is such that $E(\cG)\in \duvw$ and that there exists a non-decreasing $\varphi\in\convfs$ such that 
\beq\label{eq: hip igual}
 P_\varphi^\cV(E(\cG^{\rm op}))=P_\varphi^\cV(E(\cG))\,.
 \eeq
Let $B\in L(L^2(\R^k))^+$ be SP, with $R(B)\subset \cV$ and such that  
$S_{E(\cG)}=S_{E(\cF)^\#_\cV}+B = \Sast +B$. Then,
\ben
\item $\sum_{i\in\In}\|g_i\|^2=w$;
\item There exist $c>0$ and measurable vector fields $v_i:\T^k\rightarrow \ell^2(\Z^k)$ for $i\in\In$ such that $\{v_i(x)\}_{i\in\IN{d(x)}}$ is an ONB of $J_\cV(x)$ for a.e. $x\in\text{Spec}(\cV)$, 
$$
[ S_{E(\cF)^\#_\cV}]_x= \Bx =\sum_{i\in \IN{d(x)}} \lambda_i(\Bx) \ v_i(x)\otimes v_i(x)\ , \quad \text{ for a.e. } x\in \text{Spec}(\cV)$$ and such that for a.e. $x\in\text{Spec}(\cV)$ we have that 
$$
[B]_x
= \sum_{i=r(x)+1}^{d(x)} \big(c- \lambda_{i}(\Bx)\,\big)^
+ \ v_{i}(x)\otimes v_{i}(x) \ ,
$$
where $r(x)=\max\{2d(x)-n,\, 0\}$, for $x\in\T^k$.
\een
The constant $c\ (=c(\frac{w'}{\nu(Y)}))>0$ does not depend on $\cG$. Moreover, in this case $P_\psi^\cV(E(\cG))=P_\psi^\cV(E(\cG^{\rm op}))$ for every non-decreasing $\psi\in\convf$.
\end{cor}
\begin{proof}
We use the notions and notations from the proof of Theorem \ref{teo opt dual con norm}. 
Arguing as in the last part of the proof of Theorem \ref{teo opt dual con norm} we see that Eq. \eqref{eq: hip igual} implies that 
$$
\int_Y \varphi\circ h \ d\tilde \nu= \int_Y \varphi\circ f_c\ d\tilde \nu\, ,
$$ where $\tilde \nu=\nu(Y)^{-1}\, \nu$ is the probability measure obtained by normalization of $\nu$.
By Corollary \ref{pregunta reord1.5} we get that $h 
=f_c\,$, where $c=c(\frac{w'}{\nu(Y)})$ is as in the proof of Theorem \ref{teo opt dual con norm}. 
Therefore, for $i\in\In$ and $x\in X_i$, then 
\beq \label{eq: rel entre especs}
 \lambda_{i-j+1}(\Bx)+\lambda_{j}([B]_x)=
 \left\{
  \begin{array}{ccc}
    \max\{ \lambda_{i-j+1}(\Bx) ,\, c\}  & {\rm if} & 1\leq j\leq r_i\,;  \\
       \lambda_{i-j+1}(\Bx) & {\rm if} & r_i+1\leq j\leq i\,.  \\
   \end{array}
	\right.
 \eeq
Moreover, by Eq. \eqref{eq: preparando la cosa} and the properties of $h$ we have that 
\begin{eqnarray*}
P_\varphi^\cV(E(\cG))&=& \sum_{i\in\In}\int_{X_i} \sum_{j\in\IN{i}} 
\varphi(\lambda_j(\,[ S_{E(\cG)}]_x)\,) \ dx
\\ &\geq & \int_Y \varphi\circ h \ d\nu
+\sum_{i\in\In} \int_{X_i} \sum_{j=r_i+1}^i \varphi\circ \lambda_{i-j+1}(\Bx)\ dx 
=P_\varphi^\cV(E(\cG))\ .
\end{eqnarray*}
Therefore, we should have equality Eq. \eqref{eq: preparando la cosa} for a.e. $x\in X_i$ and $i\in\In\,$. Since $\varphi$ is strictly convex, then the majorization relation in Eq. \eqref{eq: Lidskii ineq fibers}  together with the case of equality in Lidskii's inequality (see the Appendix section in \cite{MRS14b}) imply that for $i\in\In$ and a.e. $x\in X_i$ there exists an ONB $\{z_j(x)\}_{j\in\IN{i}}$ of $J_\cV(x)$ (but not necessarily of measurable vector fields as functions of $x$) such that 
\beq \label{eq: igualdad en Lidskii}
[ S_{E(\cF)^\#_\cV}]_x=\sum_{j\in\IN{i}} \la_j(\Bx)
\ z_j(x)\otimes z_j(x) \ \text{ and } \ 
[B]_x= \sum_{j\in\IN{i}} \la_{i-j+1}([B]_x)\ z_j(x)\otimes z_j(x)\, . 
\eeq
 Let $P\in L(L^2(\R^k))^+$ denote the orthogonal projection onto $\cR=\overline{R(B)}$, so that $P$ is SP and $[ P]_x = P_{R([B]_x)}$ for every $x\in \T^k$. Let $p:\T^k\rightarrow \N$ be the measurable function given by $p(x)=\tr(\,[ P]_x)$ for $x\in\T^k$. Then, by inspection of Eqs. \eqref{eq: rel entre especs} and \eqref{eq: igualdad en Lidskii} we see that $P \,S_{E(\cF)^\#_\cV}= S_{E(\cF)^\#_\cV}\,P$,
$$
[ P]_x\,\Bx\, [ P]_x + [B]_x
=c\cdot [ P]_x \ , \quad \text{ for a.e. } x\in\T^k 
$$ and,  for $i\in\In$ and $x\in X_i\,$, 
$$  
[I-P]_x\, \Bx\, [I-P]_x
=\sum_{j=1}^{i-p(x)} \la_j(\Bx)\ z_j(x)\otimes z_j(x)\,.
 $$
 Since $a +(c-a)^+=\max\{a,\,c\}$ for $a,\,c\geq 0$, 
 these last facts imply the existence of measurable vector fields $v_i:\T^k\rightarrow \ell^2(\Z^k)$ for $i\in\In$ with the desired properties; indeed, the previous identities show that we just have to consider measurable fields of eigenvectors
of the operators $P\, S_{E(\cF)^\#_\cV}\, P$ and $(I-P) \, S_{E(\cF)^\#_\cV}\, (I-P)$, whose existence follow from Lemma
\ref{lem spect represent ese}. 

\pausa
Finally, if $\cG^{\rm op}$ is as in Theorem \ref{teo opt dual con norm}, a careful inspection of the proof of that theorem shows that 
$$
\lambda(\,[ S_{E(\cG)}]_x)
=\lambda(\,[ S_{E(\cG^{\rm op})}]_x)\ , \quad\text{ for a.e. }x\in\text{Spec}(\cV)\,,$$ which implies the optimality properties of $E(\cG)$ for a non-decreasing $\psi\in\convf$.
\end{proof}

\pausa
Notice that with the notations of Corollary \ref{coro unico} above, we see that for a.e. $x\in\T^k$ then
$$
[ S_{E(\cG^{\rm op})}]_x= \sum_{i=1}^{r(x)}\lambda_i(\Bx) \ v_i(x)\otimes v_i(x) + \sum_{i=r(x)+1}^{d(x)} \max\{\lambda_i(\Bx),\, c \} \ v_i(x)\otimes v_i(x)\, ,
$$where we have used that $a+(c-a)^+=\max\{a,\,c\}$ for $a,\,c\geq 0$. In particular, notice that 
$$
\lambda_{d(x)} [ S_{E(\cG^{\rm op})}]_x \geq  \max\{ c\coma  \lambda_{d(x)} ( [ A]_x) \}$$
which implies that the condition number of $[S_{E(\cG^{\rm op})}]_x$ is smaller than or equal to the condition number of 
$[ A]_x=[S_{E(\cF)^\#_\cV}]_x$ - both acting on $J_\cV(x)$ - for a.e. $x\in\T^k$. That is, the optimal oblique dual $E(\cG^{\rm op})$ improves the (spectral) stability of the canonical oblique dual $E(\cF)^\#_\cV=E(\cF^\#_\cV)$.

\pausa
 The representation of $[S_{E(\cG^{\rm op})}]_x$ above motivates the following construction, which also characterizes
 all elements of $\duvw$ which are minimal in the sense of Theorem \ref {teo opt dual con norm}.

\begin{fed}[Non-commutative water-filling at level $c$ in $U_\cV(E(\cG)\,)\,$]\label{def: non commut wf}
\rm 
Let $\cG=\{g_i\}_{i\in\In}$ be such that $E(\cG)$ is a frame for $\cV$ 
with frame operator $\Sast=S_{E(\cG)}\,$. By Lemma \ref{lem spect represent ese} we can consider 
 measurable vector fields $v_j:\T^k\rightarrow \ell^2(\Z^k)$ for $j\in\In$ such that 
\beq\label{def: the repre}
\Bx
=\sum_{j\in \IN{d(x)}} \lambda_j(\Bx )  \ v_j(x)\otimes v_j(x) 
\eeq
is a spectral representation of $\Bx\, $, where $\{v_j(x)\}_{j\in\IN{d(x)}}$ is an ONB of $J_\cV(x)$ (here $d(x)=\dim J_\cV(x)\leq n$), for a.e. $x\in\T^k$.

\pausa
Given $c\geq 0$ then we define the (non-commutative) water-filling of $\Sast$ at level $c$ 
with respect to the representation in Eq. \eqref{def: the repre}, denoted 
$\Sast_c\in U_\cV(E(\cG)\,)$, as the unique
positive SP operator such that  operator $R(\Sast_c)\subset \cV$ and
\beq \label{eq: defi ncwf}
[ \Sast_c](x)\igdef [ \Sast_c]_x=\sum_{i\in \IN{r(x)}} \lambda_i(\Bx) \ 
v_i(x)\otimes v_i(x)  + \sum_{i=r(x)+1}^{d(x)} \max\{\lambda_i(\Bx),\, c\} \ v_i(x)\otimes v_i(x)
\eeq 
where $r(x)=\max\{2d(x)-n,0\}$ (recall $\IN{0}=\emptyset$) for a.e. $x\in\T^k$. 
\EOE\end{fed}

\begin{rem}With the notations of Definition \ref{def: non commut wf}:
\begin{enumerate}
\item We point out that $[ {\Sast_c}]$ as described in Eq. \eqref{eq: defi ncwf} is a well defined measurable field of positive semidefinite operators that is essentially bounded.
\item Notice that in the spectral representation of $[ \Sast_c](x)$ given in Eq. \eqref{eq: defi ncwf}, the eigenvalues are not necessarily arranged in non-increasing order.
\item Finally notice that $\Sast_c\in U_\cV(E(\cG)\,)$, since 
$[\Sast_c](x) - \Bx$ is a positive operator with rank at most $d(x)-r(x)\leq n-d(x)$, for a.e. 
$x\in\T^k$. \EOE
\end{enumerate}

\end{rem}

\pausa
We end this section with the following comments: with the notions and notations of Theorem \ref{teo opt dual con norm}, let 
$\Sast =S_{E(\cF)^\#_\cV}$ and 
consider measurable vector fields $v_i:\T^k\rightarrow \ell^2(\Z^k)$, for $i\in\In$, such that 
\beq\label{def: the repre2}
 \Bx =\sum_{i\in \IN{d(x)}} \lambda_i(\Bx) \ v_i(x)\otimes v_i(x) 
\eeq
is a spectral representation of $\Bx$ with respect to an eigen-basis $\{v_i(x)\}_{i\in\IN{d(x)}}$, where $d(x)=\dim(J_\cV(x))$, for a.e. $x\in\T^k$. Let $c>0$ be such that, if $\Sast_c$ is the water-filling of $\Sast$ at level $c$ with respect to the representation in Eq. \eqref{def: the repre2} then, 
$$ 
\int_{\T^k} \tr( [ \Sast_c](x)) \ dx=w\ .$$
By construction $\Sast_c\in U_\cV(E(\cF)^\#_\cV)$ and therefore, by Proposition 
\ref{pro SD}, there exists 
$\cG_0\in \duvw$ such that $S_{\cG_0}=\Sast_c\,$. As we have already noticed, $\cG^{\rm op}$ from Theorem \ref{teo opt dual con norm} is constructed in this way;
hence, in this case we have that for every non-decreasing $\varphi\in\convf$,
$$
P_\varphi^\cV(E(\cG_0))\leq P_\varphi^\cV(E(\cG)) \ , \quad \text{ for every } \cG\in \duvw\,.
$$ Moreover, by Corollary \ref{coro unico}, any structural optimal frame $\cG\in \duvw$ (i.e. such that $\cG$ is a $P_\varphi^\cV$-minimizer in $\duvw$ for every $\varphi\in\convf$) is obtained in this way. That is, the structural optimal SG $\cV$-dual frames for $E(\cF)$ with norm restrictions are
exactly those $\cG\in \duvw$ for which their frame operators are obtained in terms of the non-commutative water-filling construction from Definition \ref{def: non commut wf}.

\section{Appendix: spectral structure of $U_\cV(E(\cG)\,)$}\label{Apendixity}

In what follows we consider a measure space $(X,\mathcal X,\rho)$, such that $X\subset \T^k$ is a Lebesgue measurable set, $\mathcal X$ denotes the $\sigma$-algebra of Lebesgue sets in $X$ and $\rho$ is the Lebesgue measure restricted to $\mathcal X$.

\begin{pro}\label{pro FP medible}
Let $G:X\rightarrow \matposn$ be a bounded measurable field of positive semidefinite matrices with associated measurable eigenvalues $\la_j:X\rightarrow \R^+$ for $j\in\In$ such that $\la_1\geq \ldots\geq \la_n$.
 Assume that the measurable functions $\beta_j:X\rightarrow \R^+$ for $j\in \IN{n-1}$ satisfy the interlacing conditions
 \beq\label{ecuac interlac}
 \la_j(x)\geq \beta_j(x)\geq \la_{j+1}(x)  \quad \quad x\text{-a.e. for}\quad j\in\IN{n-1}\ .
 \eeq
 Then there exists a measurable map $W:X\rightarrow \cM_{n,\,n-1}(\C)$ such that 
$W^*(x)\,W(x)=I_{n-1} $ and 
 \beq\label{ec hay entrelace}
\la(W^*(x)\, G(x)\, W(x)\,)=(\beta_1(x),\ldots,\beta_{n-1}(x)\,)  
\ , \peso{for a.e.}  x\in X \ .
\eeq
\end{pro}
\begin{proof} 
We argue by induction on $n$  (the size of $G$). Notice that 
$\beta_1\geq \ldots\geq \beta_{n-1}$ by Eq. \eqref{ecuac interlac}. Using the results of  \cite{RS95}, 
we can consider measurable vector fields
$u_j:X\rightarrow \C^n$ for $j\in\In$ such that $\{u_j(x)\}_{j\in\In}$ is an ONB of eigenvectors of $G(x)$ for a.e. $x\in X$.

\pausa 
Assume first that $\beta_{n-1}=\la_n$. Set $G'(x)=V(x)^* \, G(x)\,V(x)$ where $V(x)$ is the $n\times (n-1)$ matrix whose columns are the vectors $u_1(x),\ldots,u_{n-1}(x)$, for $x\in X$. Then, $G'$ is a bounded measurable field of (diagonal) positive semidefinite matrices of size $n-1$ with measurable eigenvalues $\la_j:X\rightarrow \R^+$ for $j\in\IN{n-1}$. If we assume that we can find  a measurable function $Z:X\rightarrow \cM_{n-1,\,n-2}(\C)$ such that $Z^*(x)Z(x)=I_{n-2}$ and 
$\la(Z^*(x) G'(x)\, Z(x))=(\beta_1(x),\ldots,\beta_{n-2}(x))$ for a.e. $x\in X$,  we let 
$$ W(x)=\begin{pmatrix}
Z(x) & 0_{n-1}\\ 0_{n-2}^t
& 1 \end{pmatrix} \ , \quad \text{ for } x\in X\,.
$$ Then, it is easy to see that $W:X\rightarrow \cM_{n,n-1}(\C)$ has the desired properties.
By iterating the previous argument and considering a convenient partition of $X$ into measurable sets, we can assume without loss of generality that 
$$ 
 \la_j(x)>\beta_j(x)>\la_{j+1}(x) \ ,  \quad \text{for a.e. } x\in X\ , \quad j\in\IN{n-1}\,.
$$
In this case we set 
$$
\gamma_j(x)=\frac{\prod\limits_{i\in\IN{n-1}}(\lambda_j(x)-\beta_i(x))}
{\prod\limits_{k\neq j}(\lambda_j(x)-\lambda_k(x))} \ , \quad \text{for } x\in X \ , \quad j\in\In\, .$$ 
The previous assumptions (strict interlacing inequalities) imply that $\gamma_j(x)> 0$ is defined for a.e. $x\in X$; moreover, the functions $\gamma_j:X\rightarrow \R^+$ are measurable for $j\in\In$.

\pausa 
Set $\xi_j=\gamma_j^{1/2}:X\rightarrow \R^+$ for $j\in\In$, let $v=\sum_{j\in\In} \xi_j \, u_j:X\rightarrow \C^n$ 
and let $P:X\rightarrow \matposn$ 
given by $P(x) = I- P_{v(x)}$ (the orthogonal projection onto $\{v(x)\}\orto$, notice that $v(x)\neq 0$ a.e.). Let $p_x(t)\in\R[t]$ denote the characteristic polynomial
of $P(x)\,G(x)\,P(x)$. Then, a well known argument in terms of alternate tensor products (see \cite{Bhat}) shows that $$
p_x(t)=t\,\sum_{j\in\In}\gamma_j(x)\ \prod_{k\neq j}(t-\la_k(x)) \implies 
p_x(\la_j(x))=\la_j(x) \prod_{i\in\IN{n-1}}(\la_j(x)-\beta_i(x)) \,
$$
for a.e. $x\in X$,  $j\in\In$ and $p_x(0)=0$. Therefore,
$$
p_x(t)=t\, \prod_{j\in\IN{n-1}}(t-\beta_j) \peso{and} \sum_{j\in\In} \xi_j^2(x)=1 \ , \quad \text{for a.e. }x\in X \ ,
$$ 
by comparing the leading coefficients of the two representations of the polynomial. This last normalization condition shows, in particular, that $P(x)=I- v(x)\otimes v(x)$ for $x\in X$ a.e. and hence $P$ is a measurable function.

\pausa 
Finally, let $\{w_j:X\rightarrow \C^n\}_{j\in\In}$ be a measurable ONB of eigenvectors functions for $P$ such that $P(x)\,w_n(x)=0$ for a.e. $x\in X$. Set $W: X\rightarrow \cM_{n,\,n-1}(\C)$ such that $W(x)$ is the $n\times n-1$ matrix whose columns are the vectors $w_1(x),\ldots,w_{n-1}(x)$; then $W$ is a measurable function with the desired properties.
\end{proof}

\begin{lem}\label{eleccion medible}
Let $\la_j: X\rightarrow \R^+$ for $j\in\In$ be measurable functions such that $\la_1\geq \ldots\geq \la_n$.
 Let $d\in\I_{n-1}$ and let $\beta_j: X\rightarrow \R^+$ for $j\in \IN{d}$ be measurable functions such that $\beta_1\geq \ldots\geq \beta_d$ and such that they satisfy the interlacing inequalities 
 \beq\label{ecuac interlac2}
 \la_j(x)\geq \beta_j(x)\geq \la_{n-d+j}(x)  \ , \quad \text{for .a.e. } x\in X \ , \quad j\in\IN{d}\ .
 \eeq
Then, there exist measurable functions 
$\gamma_{i\coma j}: X\rightarrow \R^+$ for $0\leq i\leq n-d$ and $j\in\IN{n-i}$
such that:
\ben
\item $\gamma_{0\coma j}=\la_j$ for $j\in\In$ and $\gamma_{n-d\coma j}=\beta_j$ for $j\in\IN{d}$;
\item For $0\leq i\leq n-d$ then $\gamma_{i\coma j}(x)\geq \gamma_{i\coma j+1}(x)$ for $j\in\IN{n-i-1}$,
for a.e. $x\in X$;
\item For $0\leq i\leq n-d-1$ then $\gamma_{i\coma j}(x)\geq \gamma_{i+1\coma j}(x)\geq \gamma_{i\coma j+1}(x)$ for $j\in\IN{n-i-1}$,
for a.e. $x\in X$.
\een
\end{lem}
\begin{proof}We argue by (decreasing) induction in terms of $d$. Notice that the statement is trivially true if $d=n-1$. 
Assume that the result is true for $d+1$ interlacing measurable functions for some 
$d\in\IN{n-2}\,$. Given the measurable functions $\beta_j$ for $j\in\IN{d}$ as above, we 
shall construct measurable functions $\alpha_j: X\rightarrow \R^+$ for $j\in\IN{d+1}$ such that 
\beq\label{ec hip1}
\la_j\geq \alpha_j\geq \la_{n-(d+1)+j} \peso{for} j\in\IN{d+1} \py 
\alpha_j\geq \beta_j\geq \alpha_{j+1} \peso{for} j\in\IN{d}\, ,
\eeq and hence $\alpha_1\geq \ldots\geq \alpha_{d+1}\,$.  
Notice that the lemma would be a consequence of this construction and the inductive hypothesis (where the maps $\al_j$ play the role of $\gamma_{n-d+1\coma j}$ for $j \in \IN{d+1}$). 

\pausa
First notice that by the interlacing inequalities in Eq. \eqref{ecuac interlac2} we have that 
\beq \label{ec cond1}
\min \{ \la_{r+1} \coma \beta_r\}\geq \max\{\beta_{r+1}\coma \la_{n-d+r}\} \ , \quad \text{for $r\in \IN{d-1} $ \ and \ 
for a.e. }x\in X\,.
\eeq
We define $\alpha_j:X\rightarrow \R^+$, for $j\in\IN{d+1}$, as follows:
\beq \label{cons gamma}
\alpha_j:=\left\{
  \begin{array}{ccc}
     \max\{\beta_j\coma \la_{n-(d+1)+j}\} & {\rm if} & 1\leq j\leq d \, ; \\
    \min\{\beta_{d}\coma \la_{d+1}\} & {\rm if} & j=d+1 \, .\\ 
   \end{array}
	\right.
\eeq By construction the functions $\alpha_j$ are measurable, and it is easy to check (by using Eq. \eqref{ec cond1}) that 
they satisfy Eq. \eqref{ec hip1}. 
\end{proof}

\pausa
The following result is the Fan-Pall interlacing inequalities theorem for measurable fields of positive operators.
\begin{teo}\label{teo Fan-Pall medible}
Let $G: X\rightarrow \matposn$ be a bounded measurable field of positive semidefinite matrices with associated measurable eigenvalues $\la_j: X\rightarrow \R^+$ for $j\in\In$ such that $\la_1\geq \ldots\geq \la_n$.
 Let $d\in\I_{n-1}$ and let $\beta_j: X\rightarrow \R^+$ for $j\in \IN{d}$ be measurable functions such that $\beta_1\geq \ldots\geq \beta_d$. Then the following conditionas are equivalent:   
\ben
\item  \label{ecuac interlac teo}
$\la_j(x)\geq \beta_j(x)\geq \la_{n-d+j}(x)\ ,   \quad \text{for a.e. }x\in X\ , \quad j\in\IN{d}\, $.
\item \label{ec hay entrelace teo} There exists a projection valued measurable function $P: X\rightarrow \matposn$ such that 
$$
\rk\,P(x)=d \peso{and} \la(P(x)\, G(x)\, P(x))=(\beta_1(x),\ldots,\beta_{d}(x),0_{n-d}) \ , \quad \text{for a.e. }x\in X\,.
$$
\een
\end{teo}
\begin{proof} Assume first that the functions $\{\beta_j\}_{j\in\IN{d}}$ satisfy the interlacing 
inequalities in item \ref{ecuac interlac teo}. Let $\gamma_{i\coma j}: X\rightarrow \R^+$ for $0\leq i\leq n-d$ and 
$j\in\IN{n-i}$
be measurable functions as in By Lemma \ref{eleccion medible}. 
By Proposition \ref{pro FP medible} there exists a measurable function 
$W_1: X\rightarrow \cM_{n \coma \,n-1}(\C)$ such that 
$$ 
W_1(x)^*W_1(x)=I_{n-1}\py  \la(W_1(x)^* G\,W_1(x))=(\gamma_{1\coma 1}(x),\ldots,\gamma_{1\coma \,n-1}(x)) 
\ , \quad  \text{for a.e. }x\in X\ .
$$ 
Arguing as before, using Proposition \ref{pro FP medible} we can construct for $2\leq i\leq n-d$
measurable functions $W_i: X\rightarrow \cM_{n-i+1 \coma n-i}(\C)$ such that 
$W_i(x)^*W_i(x)=I_{n-i}$ for a.e. $x\in X$ and 
$$
\la(W_i(x)^*\cdots W_1(x)^* G(x) \,W_1(x)\cdots W_i(x))
=(\gamma_{i\coma 1}(x),\ldots,\gamma_{i\coma n-i}(x)) \ , \quad  
\text{for a.e. }x\in X\ .
$$ 
Let $W=W_1\cdots W_{n-d}: X \rightarrow \cM_{n\coma d}(\C)$  which is measurable by construction and notice that 
$$
W^*(x)W(x)=I_d \py \la(W(x)^*G(x)\,W(x))=(\beta_1(x),\ldots,\beta_d(x)) \ , \quad  \text{for a.e. }x\in X\ .
$$
Hence, if we set $P=WW^*: X\rightarrow \matn$ then $P$ is a measurable field of projections with the desired properties.

\pausa
Conversely, assume that there exists a projection valued measurable function $P: X\rightarrow \matposn$ satisfying 
item \ref{ec hay entrelace teo}. Then item \ref{ecuac interlac teo} is a straightforward consequence of the so-called 
Cauchy  interlacing inequalities from matrix analysis (see for example \cite{Bhat}).
\end{proof}

\pausa
Let $\cG=\{g_i\}_{i\in\In}$ be such that $E(\cG)$ is a frame for the SI subspace $\cV$, with frame operator $\Sast=S_{E(\cG)}$. 
Recall that (see Definition \ref{el conjunto U})
 $$U_\cV(E(\cG)\,)=\{\Sast + B:  B\in L(L^2(\R^k))^+ \text{ is SP},\ R(B)\subset \cV,\
\rk([B]_x)\leq n-d(x) \, , \, \text{for a.e. } x\in \T^k \}\, . 
$$
Using the Fan-Pall inequalities for measurable fields of matrices we can now describe the fine spectral structure of the elements in $U_\cV(E(\cG)\,)$

\begin{teo}\label{estructdelU}
Let $\cV$ be a SI subspace in $L^2(\R^k)$ with $\text{Spec}(\cV)\inc X$ and let $\cG=\{g_i\}_{i\in\In}$ 
be such that $E(\cG)$ is a frame for $\cV$ with frame operator $\Sast=S_{E(\cG)}$.  Let $d: X\rightarrow \N$ 
be the measurable function given by $d(x)=\dim J_\cV(x)$, for $x\in \text{Spec}(\cV)$, and let 
$m(\cdot)=2d(\cdot)-n$. Given a measurable function $\mu:X\rightarrow (\ell^1_+(\N)^+)\da$ (decreasing sequences)
 the following are equivalent:
 \ben
 \item There exists $C\in U_\cV(E(\cG)\,)$ such that 
$\mu (x) =  \la (\Bx)$  for a.e. $x\in X$;

\item $\mu(x) =0$ for every $x\notin\text{Spec}(\cV)$.  If $x\in\text{Spec}(\cV)$ then   $\mu_i(x)=0$ for $i\geq  d(x)+1$ and 
\begin{enumerate}
\item in case that $m(x)\leq 0$, then  $\mu_i(x)\geqp  \la_i(\Bx) $ for  $i\in \I_{d(x)}$;
\item in case that $m(x)\ge 1$, then  $\mu_i(x)\geqp  \la_i(\Bx)$ for $i\in \I_{d(x)}$
and  
$$
\mu_{n- d(x)+i}(x) = \mu_{d(x)-m(x)+i}(x)\le  \la_i(\Bx) \peso{for} i\in \I_{m(x)}  \ .
$$
\een
 \een
\end{teo}
\begin{proof} First notice that by considering a convenient finite partition of $X$ into measurable sets we can assume, without loss of generality, that $d(x)=d\in \N$ for a.e. $x\in \text{Spec}(\cV)$.   

\pausa
Let $C\in U_\cV(E(\cG)\,)$, and assume that 
$\mu=\lambda(\, [C])$. By hypothesis, there exists
$B\in L(L^2(\R^k))^+$ SP, with $R(B)\subset \cV$, 
$\rk([B]_x)\leq n-d(x)$ for a.e. $x\in \T^k$, such that $C=\Sast+B$.	
By Lemma \ref{pro factoriz} there exists $\cZ=\{z_i\}_{i\in\In}\in\cV^n$ such that $ T_{E(\cG)}\,T_{E(\cZ)}^*=0$ and  $B=S_{E(\cZ)}$.		If we let $\cG+\cZ=\{g_i+z_i\}_{i\in\In}$ then $T_{E(\cG+\cZ)}=T_{E(\cG)}+T_{E(\cZ)}$ 
and 
$$
S_{E(\cG+\cZ)}= T_{E(\cG+\cZ)} \, T_{E(\cG+\cZ)}^*=
S_{E(\cG)}+ S_{E(\cZ)}= \Sast + S_{E(\cZ)} = C
$$
 with $R([ C]_x)=J_\cV(x)$ and $\dim J_\cV(x)=d\leq n$. Then, 
$$
\lambda_j\big(\, [T_{E(\cG+\cZ)}^*\  T_{E(\cG+\cZ)}]_x\big)=\mu_j(x) \quad \text{ for $j\in\IN{d}$ and a.e. } x\in\text{ Spec}(\cV)\ .
$$
Moreover, if we let $P:\text{ Spec}(\cV)\rightarrow \matn^+$ be the projection valued measurable function 
such that $P(x)$ is the orthogonal projection onto span $\{\Gamma \cG(x)\}
=R(T_{\Gamma \cG(x)}^*)$ then, using again that 
$ T_{E(\cG)}\,T_{E(\cZ)}^*=0$ we see that 
$$ P(x) \ (\, [T_{E(\cG+\cZ)}^*\  T_{E(\cG+\cZ)}]_x \, )\ P(x)
=[T_{E(\cG)}^*\  T_{E(\cG)}]_x\ , \quad \text{ for a.e. } x\in\text{ Spec}(\cV)\,.$$
Since $\rk(P(x))=d\leq n$ and 
$$ 
\lambda_j(\,[T_{E(\cG)}^*\  T_{E(\cG)}]_x)
= \lambda_j(\,[T_{E(\cG)}\ T_{E(\cG)}^* ]_x)=\lambda_j(\Bx)
\ , \quad \text{ for $j\in\IN{d}$ and a.e. } x\in\text{ Spec}(\cV) 
$$ 
then, using Theorem \ref{teo Fan-Pall medible} we conclude that
that the Fan-Pall inequalities hold between 
$$
(\mu_1(x),\ldots,\mu_d(x),0_{n-d}) \quad  \text{and} \quad 
(\lambda_1(\Bx)\coma \ldots \coma \lambda_d(\Bx) \coma 0_{n-d})\ .
$$  
A careful inspection of these inequalities for the previous vectors shows that the inequalities in items 2.a and 2.b. above hold (according to the relation between $n$ and $d$).

\pausa
Conversely, let $\mu:\text{Spec}(\cV)\rightarrow \ell^1(\N)^+$ satisfy the conditions in item 2 and let $D_\mu(\cdot):\text{Spec}(\cV)\rightarrow \matn^+$ be the measurable field of positive semidefinite matrices such that $D_\mu(x)$ is the diagonal matrix with main diagonal $(\mu_1(x),\ldots,\mu_d(x),0_{n-d})$ for $x\in X$. Then, by Theorem \ref{teo Fan-Pall medible} there exists a projection valued measurable function $P: \text{Spec}(\cV)\rightarrow \matposn$ such that   $$
\tr(P(x))=d \peso{and} \la(P(x) \ D_\mu(x)\ P(x))
=(\lambda_1(\Bx)\coma \ldots \coma \lambda_d(\Bx) \coma 0_{n-d})
\in(\R^+)^n\ .
$$ 
In this case we see that 
$$
\lambda(D_\mu^{1/2}\, P(x)\,D_\mu^{1/2})
=(\lambda_1(\Bx)\coma \ldots \coma \lambda_d(\Bx) \coma 0_{n-d}) \in(\R^+)^n\,.
$$
Let $D_\lambda(x)$ be the diagonal  matrix with main diagonal 
$(\lambda_1(\Bx)\coma \ldots \coma \lambda_d(\Bx) \coma 0_{n-d})$. 
By taking an appropriate measurable field of unitary matrices 
$U(x):\text{Spec}(\cV)\rightarrow \matn$ we conclude that
 \beq\label{eq: falta poquito1} 
D_\lambda(x) =U(x)^*
\Big(D_\mu^{1/2}\, P(x)\,D_\mu^{1/2} \Big)
\,U(x) \ ,\quad \text{ for a.e. } x\in\text{ Spec}(\cV)  \, .
   \eeq
Arguing as in the proof of Lemma \ref{pro factoriz} we see that there exist measurable fields of vectors $v_j:\text{Spec}(\cV)\rightarrow \ell^2(\Z^k)$, for $j\in\IN{d}$, such that $\Bx\,v_j(x)=\lambda_j(\Bx)\ v_j(x)$ and  $\cB(x)=\{v_j(x)\}_{i\in\IN{d}}$ is an ONB of $J_\cV(x)$ for a.e. $x\in \text{Spec}(\cV)$.
We finally consider $B\in L(L^2(\R^k))^+$ S.P. with $R(B)\subset \cV$, uniquely determined by the condition:
\beq\label{eq: falta poquito2} 
\{[B]_x\}_{\cB(x)}= U(x)^*
\Big(D_\mu^{1/2}\, (I-P(x))\,D_\mu^{1/2}\Big)
\,U(x) \ ,\quad \text{ for a.e. } x\in\text{ Spec}(\cV)  \, 
\eeq 
where $\{[B]_x\}_{\cB(x)}$ stands for the matrix representation of $[B]_x$ with respect to the ONB $\cB(x)$ of $J_\cV(x)$; in particular, using that $\rk(I-P(x))=n-d$, we conclude that $\rk([B]_x)\leq n-d$ for a.e. $x\in\text{Spec}(\cV)$.
 On the other hand, by construction of $\cB(x)$, we have that $\{\Bx\}_{\cB(x)}=D_\lambda(x)$: thus, using Eqs \eqref{eq: falta poquito1} and  \eqref{eq: falta poquito2} we have that 
$$ \{\Bx +[B]_x\}_{\cB(x)}= \{\Bx\}_{\cB(x)} + \{[B_x]\}_{\cB(x)}=U(x)^* \ D_\mu\ U(x) \ ,\quad \text{ for a.e. } x\in\text{ Spec}(\cV) \,. $$
This last fact implies $C=\Sast+B\in U_\cV(E(\cG)\,)$ satisfies that $\lambda_j([ C]_x)=\mu_j(x)$ for $j\in\N$ and a.e. $x\in\text{Spec}(\cV)$.
\end{proof}

\end{document}